\documentclass[12pt,a4paper,oneside]{amsart}
\usepackage[left=1.1in, right=0.8in, top=1.0in, bottom=1.0in]{geometry}
\usepackage{graphicx}
\usepackage{setspace}
\usepackage{indentfirst} 
\usepackage{cases}
\usepackage{wrapfig}
\usepackage[toc,page]{appendix}
\usepackage{amsmath}
\usepackage{amsthm}
\usepackage{amssymb}
\usepackage{enumerate}
\usepackage{mathrsfs}
\usepackage{dcolumn}
\usepackage{xcolor}
\usepackage{mathtools}

\newcolumntype{L}{D{.}{.}{2,5}}
\theoremstyle{plain}
\newtheorem{thm}{Theorem}[section]
\newtheorem{lemma}[thm]{Lemma}
\newtheorem{remark}[thm]{Remark}
\newtheorem{mydef}[thm]{Definition}

\newtheorem{corollary}[thm]{Corollary}
\newtheorem{proposition}[thm]{Proposition}
\newtheorem{question}[thm]{Question}

\DeclareMathOperator{\Id}{Id}

\DeclareMathOperator{\inte}{int}

\DeclareMathOperator{\CAT}{CAT}

\DeclareMathOperator{\diam}{\text{diam}}

\theoremstyle{definition}
 \usepackage{hyperref}
 \hypersetup{
    colorlinks=true,
    linkcolor=black,
    filecolor=magenta,      
    urlcolor=cyan}

\begin{document}
	\title[Existence and uniqueness of optimal transport maps in $\CAT(0)$ spaces]{Existence and uniqueness of optimal transport maps in locally compact $\CAT(0)$ spaces}
	\author{{A. B\"erd\"ellima}}
	
		\AtEndDocument{\bigskip{%
		  \textsc{Berlin 10709, Germany.} }}
	\maketitle
	\begin{abstract}
	We show that in a locally compact complete $\CAT(0)$ space satisfying positive angles property and a disintegration regularity for its canonical Hausdorff measure, there exists a unique optimal transport map that push-forwards a given absolutely continuous probability measure to another probability measure. In particular this holds for the Riemannian manifolds of non-positive sectional curvature and $\CAT(0)$ Euclidean polyhedral complexes. Moveover we give a polar factorization result for Borel maps in $\CAT(0)$ spaces in terms of optimal transport maps and measure preserving maps.
	\end{abstract}

\section{Introduction}
The Monge problem \cite{Monge} asks for a distribution of mass to be moved to another one in such a way that the average cost of transportation is minimized. This transportation rule if it exists, it is known as an optimal transport map. However an optimal transport map does not always exist. Kantorovich \cite{Kantorovich} proposed a relaxation of the problem which always guarantees an optimal (not necessarily unique) solution, referred to as an optimal transport plan. When Monge problem has a solution then so does Kantorovich problem and the two solutions essentially coincide. One can formulate both problems in general for any complete separable metric space $(X,d)$. More specifically given a cost function $c:X\times X\to(-\infty,+\infty]$ and two (probability) measures $\mu$ and $\nu$ on $X$ the Monge problem asks to find
\begin{align}
\label{eq:Monge-problem}
T^*\in\arg\min\int_Xc(x,Tx)\,d\mu(x)\quad\text{(Monge formulation)}
\end{align}
among all Borel mappings $T:X\to X$ that push-forward $\mu$ to $\nu$, denoted by $T_{\#}\mu=\nu$, that means for every continuous function $\phi$ on $X$ 
$$\int_X\phi(T(x))\,d\mu(x)=\int_X\phi(y)\,d\nu(y).$$
Analogously one can cast Kantorovich problem as a minimization problem
\begin{align}
\label{eq:Kantorovich}
\pi^*\in\arg\min\int_Xc(x,y)\,d\pi(x,y)\quad\text{(Kantorovich formulation)}
\end{align}
among all transport plans $\pi$ from $\mu$ to $\nu$, that is
$\pi(A\times X)=\mu(A)$ and $\pi(X\times B)=\nu(B)$
for all Borel sets in $A,B\in X$. Often the cost function $c(x,y)$ is taken to be $d(x,y)^2/2$.
Monge problem has been extensively studied by many authors initially in the setting of a Euclidean space. Sudakov \cite{Sudakov} was the first to show solutions of the problem as mappings from $\mathbb R^n$ to $\mathbb R^n$ by using a method of decomposition of measures. A different existence proof is due to Evans--Gangbo \cite{Evans} which employs tools from theory of partial differential equations. Another approach is provided by Cafarrelli--Feldman--McCann \cite{Cafarrelli}, who consider a more general cost function $c(x,y)=\|x-y\|^p$ for $p\geq 1$. They apply
 a change of coordinates that adapts the local geometry of the problem so that one needs only solve one dimensional transportation problems, a method  that was  independently discovered and used by Trudinger--Wang \cite{Wang} .
 This approach is used again by Feldman--McCann \cite{Feldman} to show existence and uniqueness results for the case when the underlying space is a Riemannian manifold. Previously McCann \cite{McCann} had proved similar theorems, but for compact manifolds, and as a result was able to show the polar factorization of a Borel function and its relationship to Helmholtz--Hodge decomposition of a vector field on a manifold. This generalized a theorem first established by Brenier \cite{Brenier} in the setting of a Euclidean space. Brenier also showed that the solution of the polar factorization problem coincides with that of Monge--Kantorovich problem, see Brenier's monograph \cite{Brenier91}. Later the transport problem was carried over to more general metric spaces with one of the earliest works being that of \cite[Lott--Villani]{Lott}, where they use optimal transport maps to give a notion of lower Ricci curvature for a measured length spaces. Related work in metric measure spaces can be found in \cite[Figalli--Villani]{Figalli}, \cite[Cavalleti--Huesmann ]{Huesman} and a series of papers by \cite[Sturm]{Sturm1, Sturm2, Sturm3}. 
 For an extensive treatment on the history of Monge--Kantorovich problem, its applications and generalizations we refer to \cite[Ambrosio--Gigli--Savar\'e]{Ambrosio}, \cite[Villani]{Villani} and the references therein.
 
 Following these developments we attempt in this work to obtain existence and uniqueness results for the transportation problem in the setting of a locally compact complete $\CAT(0)$ space. These spaces comprise a special class among metric spaces with bounded curvature introduced by Alexandrov \cite{Alexandrov}. Metric spaces with bounded curvature from above were popularized by Gromov \cite{Gromov} and play an important role in pure and applied mathematics, e.g. see \cite{Brid, Burago, BHV}. Seen as a generalization of Riemannian manifolds with non-positive curvature, $\CAT(0)$ spaces offer a natural setting for investigation of the transport problem. Moreover our work complements in a way the result of \cite[Bertrand]{Bertrand} about the transport problem in Alexandrov spaces with curvature bounded from below.
 For a different proof refer to a recent work by \cite[Rajala--Schultz]{Rajala}. 
 Apart from a lack of smoothness in general, the main difficulty arises from the fact that there does not appear an obvious way to disintegrate the associated Hausdorff measure $\mathscr H_d$ (where $d\geq 1$ is the Hausdorff dimension of $X$) consistent with some partition that is absolutely continuous with respect to the one dimensional Hausdorff measure $\mathscr H_1$. This form of disintegration of $\mathscr H_d$ is essential in utilizing the argument in Rademacher's Theorem, e.g. see \cite[Theorem 10.8]{Villani}, to show $\mathscr H_d$-almost everywhere (geodesic) differentiability of Lipschitz functions on $X$.
 But the unique geodesics structure of the space when allowed to enjoy certain conditions provides a favorable setting for the existence and uniqueness of optimal transport maps. More specifically we single out three conditions, two of which are more elementary and geometric in nature, while the third involves disintegration of $\mathscr H_d$. The dimension $d$ for simplicity is taken to be finite, though this doesn't need to be so. The first property, the local geodesic extensions, requires that in a neighborhood of each $x\in X$ any geodesic $\gamma$ containing $x$ can be extended in both directions at least incrementally. It turns out that this property when it holds everywhere is equivalent to the usual notion of geodesic extension, meaning that any geodesic can be extended indefinitely in both directions.
The geodesic extensions plays an important role in obtaining a Fermat theorem for characterization of local extrema of geodesically differentiable functions. 
The second condition demands that for every $x\in X$ there must exist some geodesic ball $B(x,r)$ for some $r>0$ such that 
the Alexandrov angle between any two distinct geodesics issuing from $x$ is strictly positive whenever the two geodesics are not subsets of one another and they are entirely contained in $B(x,r)$. We refer to this condition as the positive angles property. Though one can in principle construct $\CAT(0)$ spaces where this property fails to hold anywhere, most of the useful spaces enjoy positive angles property. The positive angles property guarantees the injectivity of the geodesic derivative $D_xc(x,\cdot\,;\gamma)$, a condition often faced in dynamical systems referred to as a twist condition.
The third property requires some regularity of the Hausdorff measure, which if satisfied, permits this measure to have a disintegration consistent with a certain partition that is absolutely continuous with respect to $\mathscr H_1$. We refer to this property as the disintegration regularity condition. It is important to say that all conditions need only hold $\mathscr H_d$-almost everywhere.
Throughout we work with the cost function $c(x,y)=d(x,y)^2/2$.

The rest of the paper develops as follows. In Section \ref{s:preliminaries} we present some preliminary definitions and results about $\CAT(0)$ spaces, measure theory, cyclic monotonicity and convexity, Hausdorff outer measure and dimension, and differentiability along geodesics of real valued functions. In Section \ref{s:3-conditions} we discuss the three conditions that guarantee existence and uniqueness result for the optimal transport map. The local geodesic extensions is satisfied $\mathscr H_d$-a.e. by any locally compact complete $\CAT(0)$ space (Theorem \ref{th:a.e.-extension}). A particular attention is given in \S \ref{ss:disintegration}, where the concepts of the radial projection and disintegration regularity are introduced. We show that if the space $(X,d)$ enjoys the positive angles property and additionally it satisfies the disintegration regularity then on every closed ball the corresponding Hausdorff measure admits a unique disintegration consistent with the partition induced by the maximal geodesics in the ball that is absolutely continuous with respect to $\mathscr H_1$ (Thereom \ref{th:H1}). Next we discuss a relationship between the Hausdorff measure of a given set $S\subseteq X$ and the positive angles property (Theorem \ref{th:boundary}). In both of these results we make use of a recent generalization \cite{Esmayli,Piotr} to metric spaces of the well known Eilenberg's inequality. In Section \ref{s:basic-lemmas} we present two fundamental lemmas (Lemma \ref{l:min-max}, Lemma \ref{l:twist}) for accomplishing the main result. Lemma \ref{l:min-max} is a type of Fermat theorem for geodesic spaces satisfying the geodesic extensions property, while Lemma \ref{l:twist} is exclusively dependent on the non-positive curvature of the space and ensures that the injectivity of $D_x(x,\cdot;\gamma)$ holds. In Section \ref{s:main} we present our main result (Theorem \ref{th:1}), where we show that in a $\CAT(0)$ space satisfying the positive angles and the disintegration regularity property $\mathscr H_d$-almost everywhere, given two probability measures $\mu,\nu$ with  $\mu$ absolutely continuous, there exists a unique optimal transport plan $\pi$ from $\mu$ to $\nu$. Moreover this map can be expressed as $\pi=(\Id, T_{\#})\mu$ for a Borel measurable map $T:X\to X$ that $\mu$-almost everywhere is unique.
 As a consequence we obtain existence and uniqueness of optimal transport map for the case when $X$ is a Riemannian manifold with non-positive sectional curvature (Theorem \ref{th:manifold}) and for the case when $X$ is a $\CAT(0)$ Euclidean polyhedral complex (Theorem \ref{th:poly-complex}). While the positive angles condition is not difficult to prove in these instances, the disintegration regularity is more involved and it makes use of a classical result of Federer \cite{Federer}. For the case of Riemannian manifolds the map $T$ takes an explicit form given by $T(x)=\exp_x\widetilde\nabla\psi(x)$ where $\exp_x:\mathscr T_x\mathscr M\to\mathscr M$ is the exponential map and $\psi$ is a $c$-convex function. Here $\widetilde{\nabla}\psi(x)$ denotes the $\mu$-approximate gradient of $\psi$ at $x$. We end with Section \ref{s:polar} where we present a polar factorization theorem (Theorem \ref{th:polarization}) for Borel measurable maps in $X$. With the help of Lemma \ref{l:invertible} it is shown that every Borel measurable map $s:X\to X$ and probability measure $\mu$ such that the push-forward of $s_{\#}\mu$ is absolutely continuous, factors uniquely $\mu$-almost everywhere into $T\circ u$ where $T$ is the unique transport map from $\mu$ to $s_{\#}\mu$ and $u$ is a measure preserving map of $\mu$, i.e. $u_{\#}\mu=\mu$.

\section{Preliminary definitions and results}
\label{s:preliminaries}
\subsection{Geometry of $\CAT(0)$ spaces}
Let $(X,d)$ be a metric space.  A curve $\gamma:[0,1]\to X$ is a constant speed geodesic if $d(\gamma(s),\gamma(t))=|s-t|\,d(\gamma(0),\gamma(1))$ for all $s,t\in[0,1]$. We denote by $\ell(\gamma)=d(\gamma(0),\gamma(1))$ the length of the geodesic $\gamma$. The metric space $(X,d)$ is a (uniquely) geodesic space if any two elements $x,y\in X$ can be connected by a (unique) geodesic in $X$. Often a geodesic segment between two elements $x,y\in X$ is denoted by $[x,y]$.
Given $x\in X$ and $r>0$ we let $B(x,r)=\{y\in X\;:\;d(x,y)<r\}$ and $B[x,r]=\{y\in X\;:\;d(x,y)\leq r\}$ denote the open and closed geodesic ball respectively. The space $(X,d)$ is locally compact if closed balls are compact. A geodesic triangle $\Delta(x,y,z)$ with vertices $x,y,z\in X$ is the union of three geodesic segments $[x,y], [y,z]$ and $[z,x]$.  A comparison triangle $\Delta(\overline{x},\overline{y},\overline{z})$ in $\mathbb R^2$ for $\Delta(x,y,z)$ is a triangle determined by three vertices $\overline{x},\overline{y},\overline{z}\in\mathbb R^2$ such that $d(x,y)=|\overline x\,\overline y|, d(y,z)=|\overline y\,\overline z|$ and $d(z,x)=|\overline z\,\overline x|$. A point $\overline w\in[\overline x,\overline y]$ is a comparison point for $w\in[x,y]$ if $d(x,w)=|\overline x\,\overline w|$. A geodesic metric space $(X,d)$ is a $\CAT(0)$ space if for every geodesic triangle $\Delta(x,y,z)$ and for any $p\in[x,y], q\in[x,y]$ the inequality holds
\begin{equation}
\label{eq:CAT-ineq}
d(p,q)\leq |\overline p\,\overline q|.
\end{equation}
An important consequence of \eqref{eq:CAT-ineq} is that any two distinct points are connected by a unique geodesic. Let $\Gamma_x(X)=\{\gamma-\text{geodesic}\,:x\in\gamma\}$ and $\Gamma(X)=\bigcup_{x\in X}\Gamma_x(X)$. We denote by $\Gamma^0_x(X)=\{\gamma-\text{geodesic}\,:\gamma(0)=x\}$. Note that $\Gamma^0_x(X)\subseteq \Gamma_x(X)$. In general given a connected set $S\subseteq X$ we define $\Gamma_x(S)=\{\gamma-\text{geodesic in}\,S\,:x\in\gamma\}$.
By $x_t:=(1-t)x\oplus ty$ we denote the unique element on $[x,y]$ satisfying $d(x_t,x)=td(x,y)$. We refer to it as the convex combination of $x$ and $y$ with parameter $t\in[0,1]$.
The $\CAT(0)$ inequality \eqref{eq:CAT-ineq} is equivalent to
\begin{equation}
\label{eq:quadratic}
d(x_t,z)^2\leq (1-t)\,d(x,z)^2+t\,d(y,z)^2-t(1-t)\,d(x,y)^2,\quad \forall z\in X.
\end{equation} 
Inequalities \eqref{eq:CAT-ineq} and \eqref{eq:quadratic} characterize the non-positive curvature of a $\CAT(0)$ space.
A set $C\subseteq X$ is convex if for any $x,y\in C$ the segment $[x,y]$ is entirely contained in $C$. 
For a set $C\subseteq X$ and any $x\in X$ we let $d(x,C)=\inf\{y\in C\,:\,d(x,y)\}$ and $P_Cx=\{y\in C\,:\,d(x,y)=d(x,C)\}$ (the metric projection onto $C$).

Given $\gamma,\eta\in\Gamma_x(X)$ with $\gamma(0)=\eta(0)=x$ and $0<s,t<1$ consider the comparison triangle $\Delta(\overline x,\overline\gamma(s),\overline\eta(t))$. Let $\overline\alpha_{\overline x}(\overline\gamma(s),\overline\eta(t))=\angle(\overline\gamma(s)\,\overline x\,\overline\eta(t))$ be the angle at $\overline x$ between the line segments $[\overline x,\overline\gamma(s)]$ and $[\overline x,\overline\eta(t)]$. The Alexandrov's (upper) angle at $x$ between the geodesics $\gamma$ and $\eta$ is defined as the unique number in $[0,\pi]$ given by
\begin{equation}
\label{eq:Alexandrov's angle}
\alpha_x(\gamma,\eta)=\limsup_{s,t\to 0}\overline\alpha_{\overline x}(\overline\gamma(s),\overline\eta(t)).
\end{equation}
Because the functions $s\mapsto \overline\alpha_{\overline x}(\overline\gamma(s),\overline\eta(t))$ and $t\mapsto \overline\alpha_{\overline x}(\overline\gamma(s),\overline\eta(t))$ are non-decreasing, due to non-positive curvature of the space, the limit superior in \eqref{eq:Alexandrov's angle} could be replaced the simpler limit $\alpha_{x}(\gamma,\eta)=\lim_{s\to 0^+}\overline\alpha_{\overline x}(\overline\gamma(s),\overline\eta(s))$, see e.g. \cite[\S 1.10-1.17]{Brid}. 
The Alexandrov angle defines a metric on the equivalence classes of geodesics emanating from a point in space. More specifically $\gamma,\eta\in\Gamma^0_x(X)$ are equivalent whenever $\alpha_x(\gamma,\eta)=0$. To each equivalence class we associate a geodesic direction $v_{\gamma}$. We let $\Sigma_x$ denote the space of directions at $x\in X$, which is the completion of the set of all geodesic directions at $x$ equipped with the Alexandrov angle metric $d_{\Sigma}(v_{\gamma},v_{\eta})=\alpha_x(\gamma,\eta)$. We denote by $T_xX=[0,+\infty)\times\Sigma_x/\{0\}\times\Sigma_x$ the tangent space at $x$ of $X$ with the metric $d_T((t,v_{\gamma}),(s,v_{\eta}))^2=t^2+s^2-2ts\cos(d_{\Sigma}(v_{\gamma},v_{\eta}))$.  
To this end we restrict ourselves to locally compact and complete $\CAT(0)$ spaces. By a standard result in analysis such spaces are always separable, thus Polish spaces. 

The following lemma collects some properties of projections onto closed convex sets:

\begin{lemma}\cite[Theorem 2.1.12]{Bacak}
	\label{l:projections}
	Let $C\subseteq X$ be a closed convex set. Then the followings are true:
	\begin{enumerate}[(i)]
		\item $P_Cx$ exists and it is unique for any $x\in X$. Moreover the following inequality is satisfied:
		\begin{equation}
		\label{eq:projection-inequality}
		d(x,P_Cx)^2+d(y,P_Cx)^2\leq d(x,y)^2,\quad \forall y\in C.
		\end{equation}
		\item For any $x'\in[x,P_Cx]$ it holds that $P_Cx'=P_Cx$.
		\item $\alpha_{P_Cx}([P_Cx,x],[P_Cx,y])\geq\pi/2\quad\text{for every}\;y\in C$.
	\end{enumerate}

\end{lemma}

\subsection{Measure theory} Let $(X,d)$ be a metric space and $\mathscr B(X)$ its Borel $\sigma$-algebra. A function (mapping) $f:X\to X$ is measurable if $f^{-1}(B)\in\mathscr B(X)$ whenever $B\in\mathscr B(X)$.
 Given two measures $\mu,\nu:\mathscr B(X)\to[0,+\infty]$ we say $\nu$ is absolutely continuous with respect to $\mu$ and we denote it by $\nu\ll\mu$, if $\mu(B)=0$ implies $\nu(B)=0$ for every $B\in\mathscr B(X)$. By the classical Radon--Nikodym Theorem this is equivalent to existence of a measurable function $f:X\to[0,+\infty)$ such that $\nu(B)=\int_Bf(x)\,d\mu(x)$ for all $B\in\mathscr B(X)$.
The triple $(X,\mathscr B(X),\mu)$ is referred to as a measure space. A measure space $(X,\mathscr B(X),\mu)$ is $\sigma$-finite if there exists countably many measurable sets $X_k$ such that $X=\bigcup_{k\in\mathbb N}X_k$ and $\mu(X_k)<+\infty$ for every $k\in\mathbb N$.
Let $$\mathscr P_2(X):=\{\mu:\mathscr B(X)\to[0,1]\,:\,\int_Xd(x_0,x)^2\,d\mu(x)<+\infty,\;\mu(X)=1\}$$ be the set of all probability measures $\mu$ on $X$ having finite second moment\footnote{Here $x_0\in X$ is some arbitrary but fixed element.}. Given a measurable mapping $f:X\to X$ we let $f_{\#}\mu$ denote the pushforward of the measure $\mu\in\mathscr P_2(X)$ under $f$, that is $(f_{\#}\mu)(B)=\mu( f^{-1}(B))$ for all $B\in\mathscr B(X)$. 

\subsection{Cyclic monotonicity and convexity} We follow terminology in \cite[\S 1.5]{Villani}. Let $c:X\times X\to(-\infty,+\infty]$ be a cost function. A set $\Gamma\subseteq X\times X$ is $c$-cyclically monotone if for any $n\in\mathbb N$ and any set of points $(x_1,y_1),\cdots,(x_n,y_n)$ in $\Gamma$ the inequality holds $\sum_{i=1}^nc(x_i,y_i)\leq\sum_{i=1}^nc(x_i,y_{i+1})$ with the convention that $y_{n+1}=y_1$. A transport plan $\pi$ is said to be $c$-cyclically monotone if it is concentrated on a $c$-cyclically monotone set $\Gamma\subseteq X\times X$, i.e. $\pi(\Gamma)=1$.
A function $\psi:X\to(-\infty,+\infty]$ is $c$-convex if $\psi\not\equiv+\infty$ and there is $\zeta:X\to[-\infty,+\infty]$ such that $\psi(x)=\sup_{y\in X}(\zeta(y)-c(x,y))$ for all $x\in X$. The $c$-transform of $\psi$ is defined as $\psi^c(y)=\inf_{x\in X}(\psi(x)+c(x,y))$ for all $y\in X$ and the $c$-subdifferential of $\psi$ is the $c$-cyclically monotone set $\partial _c\psi=\{(x,y)\in X\times X\,:\,\psi^c(y)-\psi(x)=c(x,y)\}$. Similarly the $c$-subdifferential of $\psi$ at a given $x\in X$ is given by $\partial_c\psi(x)=\{y\in X\,:\,(x,y)\in\partial_c\psi\}$.
\begin{lemma} \cite[Theorem 5.10]{Villani}
	\label{l:Kantorovich} Let $\mu,\nu\in\mathscr P_2(X)$ and $\pi\in\Pi(\mu,\nu)$. Then there exists a measurable $c$-cyclically monotone closed set $\Gamma\subseteq X\times X$ such that the followings are equivalent:
	\begin{enumerate}[(i)]
		\item $\pi$ is optimal.
		\item $\pi$ is $c$-cyclically monotone.
		\item There is a $c$-convex function $\psi$, such that $$\pi(\{(x,y)\in X\times X\,:\,\psi^c(y)-\psi(x)=c(x,y)\})=1.$$
		\item There exist functions $\psi:X\to(-\infty,+\infty]$ and $\phi:X\to[-\infty,+\infty)$, such that $\phi(y)-\psi(x)\leq c(x,y)$ for $x,y$ and equality $\pi$-a.e..
		\item $\pi$ is concentrated on $\Gamma$.
	\end{enumerate}
\end{lemma}

\begin{remark}
\label{r:c-convex}
The condition in \cite[Theorem 5.10 (ii)]{Villani} $$\inf_{\pi\in\Pi(\mu,\nu)}\int_{X\times X}c(x,y)\,d\pi(x,y)<+\infty$$ is satisfied for $c(x,y)=d(x,y)^2/2$. Moreover by \cite[Theorem 4.1]{Villani}, since $c(x,y)=d(x,y)^2/2$ is lower semicontinuous and bounded below (uniformly), there exists an optimal transport plan $\pi\in\Pi(\mu,\nu)$.
\end{remark}

\subsection{Hausdorff outer measure and dimension}
\label{ss:Hausdorff}
For a set $U\subseteq X$ define its diameter $\diam(U)=\sup\{d(x,y)\,:\,x,y\in U\}$. By convention $\diam(\emptyset)=0$.
Given $d>0$ and $S\subseteq X$ let 
\begin{equation}
\label{eq:Hausdorff}
H^d_{\delta}(S)=\inf\{\sum_{i\in\mathbb N}\diam^d(U_i)\,:\,S\subseteq\bigcup_{i\in\mathbb N}U_i,\,\diam(U_i)<\delta\}.
\end{equation}  
The $d$-dimensional Hausdorff outer measure of $S$ is defined as $\mathscr H_{d}(S):=\lim_{\delta\to 0}H^{d}_{\delta}(S)$. By Carath\'eodory's extension theorem \cite[Theorem 10.23]{Aliprantis} it determines a measure on $\mathscr B(X)$. The Hausdorff dimension $d=\dim_HS$ of $S$ is defined as $d=\inf\{d\,'\geq 0\,:\,\mathscr H_{d\,'}(S)=0\}$.
Note that in $\mathbb R^d$ the Hausdorff measure $\mathscr H_{d}$ coincides with the Lebesgue measure $\lambda$ up to a normalization constant dependent on the dimension $d$. A set $S\subset X$ is $\mathscr H_{d}$-negligible, of zero volume or simply a null set in $X$ whenever $\mathscr H_{d}(S)=0$. If $\mathscr H_d(S)>0$ then we refer to $S$ as a set of positive volume. 
An element $\mu\in\mathscr P_2(X)$ is absolutely continuous if $\mu\ll\mathscr H_{d}$.

\begin{lemma}\cite[Theorem 2]{Schleicher}
	\label{l:Hausdorff}
	The Hausdorff dimension satisfies the following properties:
	\begin{enumerate}[(i)]
		\item If $X\subset Y$ then $\dim_HX\leq \dim_HY$;
		\item If $X=\bigcup_{i\in\mathbb N}X_i$ with $\dim_HX_i\leq d$ for all $i\in\mathbb N$, then $\dim_HX\leq d$;
		\item If $X$ is countable, then $\dim_HX=0$;
		\item If $X\subset\mathbb R^d$, then $\dim_HX\leq d$;
		\item If $\dim_HX=d_1$ and $\dim_HY=d_2$, then $\dim_H(X\times Y)\geq d_1+d_2$;
		\item If $X$ is connected and contains more than one point, then $\dim_HX\geq 1$;
		\item If $f:X\to f(X)$ is a Lipschitz mapping, then $\dim_Hf(X)\leq\dim_HX$.
	\end{enumerate}
\end{lemma}

\begin{corollary}
	\label{c:Hausdorff}
	Let $(X,d)$ be a $\CAT(0)$ space, then $\dim_HX\geq 1$. If $\dim_HX=d$ and $S\subseteq X$ with $\dim_HS=d_S$ satisfies $\mathscr H_{d_S}(S)<+\infty$ then $d_S=d$ whenever $\mathscr H_d(S)>0$.
\end{corollary}

\begin{proof}
	The first claim follows directly from Lemma \ref{l:Hausdorff} (vi), since by definition any $\CAT(0)$ space is connected. For the second assertion let $\dim_HS=d_S$, then by Lemma \ref{l:Hausdorff} (i) $d_S\leq d$. Now suppose that $\mathscr H_d(S)>0$, but $d_S<d$, then
	$H^{d}_{\delta}(S)\leq \sum_{i\in\mathbb N}\diam^{d} (U_i)<\delta^{d-d_S}\,\sum_{i\in\mathbb N}\diam^{d_S}(U_i)$ where $S\subseteq\bigcup_{i\in\mathbb N}U_i$. This in turn yields $H^{d}_{\delta}(S)\leq \delta^{d-d_S}\,H^{d_S}_{\delta}(S)$. Taking limit as $\delta\to 0$ implies $\mathscr H_d(S)=0$ raising a contradiction.
\end{proof}

\begin{figure}[t]
	\includegraphics[width=10cm]{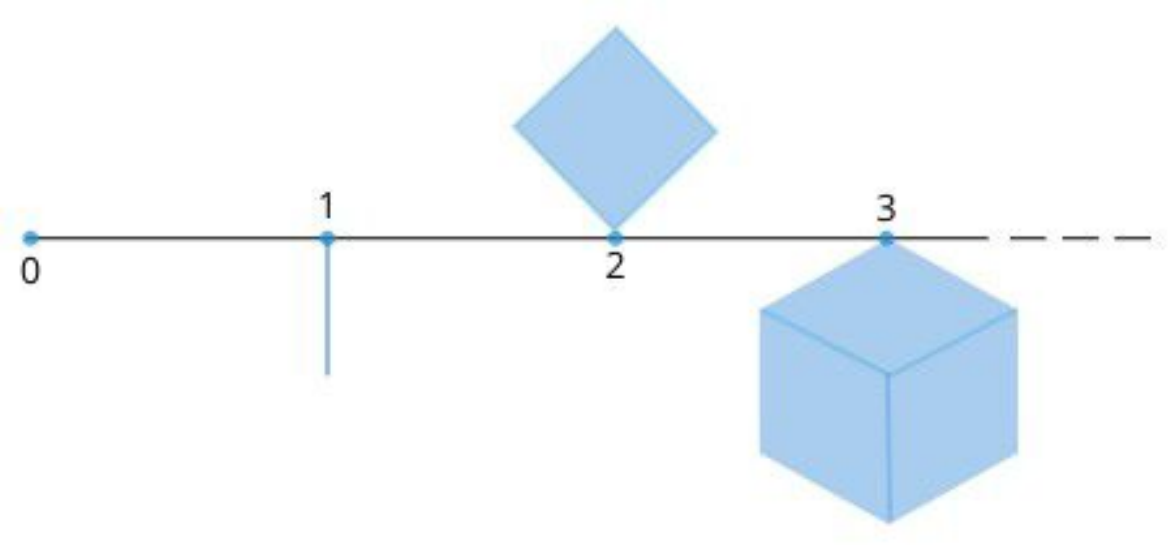}
	\caption{The space $X$ consisting of the non-negative real line $\mathbb R$ where at each point $x=n$ an $n$-dimensional solid cube is glued.}
		\label{f:line-boxes}
\end{figure}

Note that, unlike Euclidean spaces, in general a locally compact $\CAT(0)$ space may have infinite Hausdorff dimension. Consider the non-negative real line $\mathbb R_+$ where at every point $x=n$ an $n$-dimensional cube $C_n$ is glued, see Figure \ref{f:line-boxes}. Since each cube $C_n$ can be isometrically embedded in $\mathbb R^n$ then by Lemma \ref{l:Hausdorff} (iv) it follows that $\dim_HC_n=n$. Moreover again by Lemma \ref{l:Hausdorff} (i) since $C_n\subseteq X$ then $\dim_HX\geq n$ for all $n\in\mathbb N$. Also Riemannian manifolds of non-positive sectional curvature have finite Hausdorff dimension. Indeed since an $n$-dimensional manifold $\mathscr M_n$ is bi-Lipschitz diffeomorphic to $\mathbb R^n$, then Lemma \ref{l:Hausdorff} (vii) yields that $\dim_H\mathscr M_n=n$.
\begin{lemma}
	\label{l:sigma-finite}
	Let $(X,d)$ be a locally compact complete $\CAT(0)$ space with $\dim_HX<+\infty$. Let $\mathscr H_{d}$ be the Hausdorff measure on $X$. Then $(X,\mathscr B(X),\mathscr H_{d})$ is a $\sigma$-finite measure space.
\end{lemma}
\begin{proof}
	Since $(X,d)$ is locally compact and connected then it is separable.
	For a given $\varepsilon>0$ denote by $X_k=B(y_k,\varepsilon)$, where $\{y_k\}_{k\in\mathbb N}\subseteq X$ is a countable dense set, then $X=\bigcup_{k\in\mathbb N}X_k$. Next let $d_k=\dim_HX_k$, by Lemma \ref{l:Hausdorff} (i) it follows that $d_k\leq d$ and that in particular $\mathscr H_{d}(X_k)=0$ whenever $d_k<d$. So suppose that $d_k=d$, then for any $\delta>0$ it holds that $H^d_{\delta}(X_k)\leq \diam^dX_k<+\infty$, where finiteness is independent of $\delta$. Consequently $\mathscr H_{d}(X_k)<+\infty$ for every $k\in\mathbb N$. By definition $(X,\mathscr B(X),\mathscr H_{d})$ is a $\sigma$-finite space.
\end{proof}

\begin{proposition}
\label{p:finite-dim}
Every bounded closed convex set in a locally compact complete $\CAT(0)$ space has finite Hausdorff dimension. 
\end{proposition}
\begin{proof}
It suffices to prove the claim for closed geodesic balls. Let $B[x,r]$ for some $r>0$ be a closed geodesic ball in $X$. If $\dim_HB[x,r]=+\infty$ then for any $d\in[1,+\infty)$ we have $\mathscr H_d(B[x,r])=+\infty$. On the other hand for any fixed $\delta>0$ it holds that $H^d_{\delta}(B[x,r])\leq \diam^dB[x,r]<+\infty$ independently of $\delta$, implying $\mathscr H_d(B[x,r])<+\infty$. This is impossible.
\end{proof}
 Most of the theory we build on later is local in nature. Therefore
in view of Proposition \ref{p:finite-dim} one does not lose generality by assuming that $\dim_HX$ is finite. Indeed since $X$ is taken to be locally compact and complete then we can write $X=\bigcup_{k\in\mathbb N}B[x_k,r]$ where $\{x_k\}_{k\in\mathbb N}\subseteq X$ is a countable dense set and $r>0$, therefore we can equip $X$ with a volume measure $\mathscr H_{\infty}$ such that $\mathscr H_{\infty}|_{B[x_k,r]}=\mathscr H_{d_k}$ where $d_k=\dim_HB[x_k,r]$. In particular for any $x\in X$ the space locally at $x$ would be finite dimensional.

\subsection{Geodesic derivative} 
 \begin{mydef}
 	\label{d:geo-derivative}
 	Let $f:X\to(-\infty,+\infty]$ be a function and $x\in X$. For $\gamma\in\Gamma_x(X)$ define 
 	\begin{equation}
 	\label{eq:geo-derivative}
 	D^+f(x;\gamma):=\limsup_{h\to 0}\frac{f(\gamma(s+h))-f(\gamma(s))}{h}
 	\end{equation}
 	where $\gamma(s)=x$ for some $s\in[0,1]$. Similarly we define $D^-f(x;\gamma)$ with $\limsup$ replaced by $\liminf$.
 	If both limits coincide we say that $f$ is geodesically differentiable at $x$ along $\gamma$ and we simply denote it by $Df(x;\gamma)$. If $f$ is geodesically differentiable at every $x$ along every $\gamma\in\Gamma_x(X)$ then $f$ is geodesically differentiable. 
 \end{mydef}
 
 \begin{mydef}
 \label{d:geo-partial-derivative}
 Let $f:X\times\cdots\times X\to(-\infty,+\infty]$ be a function, then we denote by $D_{x_i}f(x_1,\cdots,x_n;\gamma)$ the geodesic derivative of $f$, if it exists, w.r.t. $x_i$ along $\gamma\in\Gamma_{x_i}(X)$.
 \end{mydef}

 Let $\mu$ be a measure on $X$ not assigning zero volume to positive volumes sets. A measurable set $S\subseteq X$ is said to have density $\rho$ at $x\in X$ with respect to $\mu$ if 
 \begin{equation}
 \label{eq:density}
 \lim_{\varepsilon\to 0}\frac{\mu(S\cap B(x,\varepsilon))}{\mu(B(x,\varepsilon))}=\rho.
 \end{equation}
 It is evident that $\rho\in[0,1]$ and in particular for any $\mu$-negligible set $S$ it holds that $\rho=0$ at any $x\in X$. If $X\setminus S$ is a $\mu$-negligible set then $S$ has density $\rho=1$ at any $x\in X$.
 
 \begin{mydef}
 \label{d:approx-derivative}
 Let $\Omega\subseteq X$ be an open set and $\mu$ a measure on $X$ not assigning zero volume to positive volume sets. Let $\psi:\Omega\to [-\infty,+\infty]$ be a measurable function. Then $f$ is said to be $\mu$-approximately geodesically differentiable at $x\in X$ along $\gamma\in\Gamma_x(X)$ if there exists a measurable function $\widetilde\psi:\Omega\to[-\infty,+\infty]$, differentiable at $x$ along $\gamma$, such that the set $\{\psi=\widetilde\psi\}$ has density $1$ at $x$. We define the $\mu$-approximate geodesic derivative of $\psi$ at $x$ along $\gamma\in\Gamma_x(X)$ by the formula $\widetilde D\psi(x;\gamma)=D\widetilde\psi(x;\gamma)$. 
 \end{mydef}
 
 \section{Three conditions on $\CAT(0)$ spaces}
 \label{s:3-conditions}
 \subsection{Geodesic extension}
 \begin{mydef}
 \label{d:geo-extension} Given $\gamma\in\Gamma(X)$ let $\inte\gamma=\gamma\setminus\{\gamma(0),\gamma(1)\}$. A geodesic $\gamma\in\Gamma(X)$ is extendable if there exists $\widetilde\gamma\in\Gamma(X)$ such that $ \gamma\subset\inte(\widetilde{\gamma})$. If every geodesic $\gamma\in\Gamma(X)$ is extendable, then we say that the space $(X,d)$ satisfies the geodesic extension property.
 \end{mydef}
 If $X$ has the geodesic extension property then for every $x\in X$ and $\gamma\in\Gamma_x(X)$ there exists $\widetilde\gamma\in\Gamma_x(X)$ such that $x$ is an interior point of $\widetilde \gamma$, i.e. $x=\widetilde{\gamma}(t)$ for some $t\in(0,1)$. 
 Geodesic extension can be restricted to any connected subset $S\subseteq X$ in an obvious manner, $S$ satisfies the geodesic extension property if any geodesic $\gamma\in\Gamma(S)$ is extendable in $S$. Especially open connected sets inherit this property.
 In \cite{Lytchak} geodesic extension is considered for general geodesic spaces with curvature bounded above.
 \begin{proposition}\cite[Lemma 5.8, \S2]{Brid}
 \label{p:geo-extension}
 A $\CAT(0)$ space $(X,d)$ has the geodesic extension property if and only if every geodesic can be extended indefinitely, that is for every $\gamma\in\Gamma(X)$ there exists an isometry $\mathscr I:\mathbb R\to X$, that is $d(\mathscr I(s),\mathscr I(t))=|s-t|$ for all $s,t\in\mathbb R$, such that $\mathscr I(t\,\ell(\gamma))=\gamma(t)$ for every $t\in[0,1]$.
 \end{proposition}
 Now consider the space $X=\{(x,y)\in\mathbb R^2\,:\,x\geq 0\}$ equipped with the Euclidean distance. Any geodesic starting from a point $(x,0)$ or $(0,y)$ is not extendible in the sense of Definition \ref{d:geo-extension}. However for any $(x,y)\in X$ with $x,y>0$ there is $r>0$ such that $B((x,y),r)$ enjoys the geodesic extension property. This leads to the following notion.
 
\begin{mydef}
	\label{d:extension-local}
	$(X,d)$ enjoys the local geodesic extension property at $x\in X$ if for some $r>0$ there exists $B(x,r)$ satisfying the geodesic extension property. If this condition holds for all $x\in X$ then $(X,d)$ is said to enjoy the local geodesic extension property.
\end{mydef}

 \begin{proposition}
 \label{p:equiv}
 A $\CAT(0)$ space $X$ enjoys the geodesic extensions if and only if $X$ has the local geodesic extension property.
 \end{proposition}
 \begin{proof}
 It is clear that geodesic extensions implies local extension. Now let $X$ satisfy the local geodesic extension property. Let $\gamma\in\Gamma(X)$
 and denote by 
 $x=\gamma(0), y=\gamma(1)$. There are $r,r'>0$ such that $B(x,r)$ and $B(y,r')$ satisfy the geodesic extensions. Let $\gamma^1_B=\gamma\cap B[x,r/2]$ and $\gamma^2_B=\gamma\cap B[y,r'/2]$. In particular $\gamma^1_B\in\Gamma_x(B(x,r))$ and $\gamma^2_B\in \Gamma_y(B(y,r'))$. There are $\widetilde\gamma_1\in \Gamma_x(B(x,r)), \widetilde\gamma_2\in\Gamma_y(B(y,r'))$ such that $\gamma^1_B\subset\inte\widetilde \gamma_1$ and $\gamma^2_B\subset\inte\widetilde\gamma_2$.
  But then $\gamma^*=[\widetilde\gamma_1(0),x]\cup[x,y]\cup[y,\widetilde\gamma_2(1)]$ is a geodesic in $\Gamma(X)$ strictly containing $\gamma$. 
 \end{proof}
 
  The space $X=\{(x,y)\in\mathbb R^2\,:\,x\geq 0\}$ enjoys the local geodesic extensions $\mathscr H_2$-a.e.. Indeed any point $(x,y)$ in the interior $X$  one can find an open ball at $B((x,y),r)$ such that any $\gamma\in\Gamma_{(x,y)}(B((x,y),r))$ is extendible in $B((x,y),r)$. However at points on the axis, which are $\mathscr H_2$-null sets, the local geodesic extension property fails. This raises a question of whether a general locally compact complete $\CAT(0)$ space satisfies the local geodesic extension $\mathscr H_d$-a.e., where $d$ is the Hausdorff dimension of the space. 
  
  We say an open ball $B(x,r)$ has bounded capacity if $B[x,r']$ can be covered by a finite number of balls $B(y_i,r'/2)$ whenever $r'\geq r$.

  \begin{thm}
  \label{th:a.e.-extension}
  A locally compact complete $\CAT(0)$ space $(X,d)$ with $\dim_HX<+\infty$ satisfies the local geodesic extension condition $\mathscr H_d$-a.e..
  \end{thm}
  \begin{proof}Let $\dim_HX=d$ and $x\in X$. Suppose that for any $r>0$ there is a geodesic $\gamma\in\Gamma_x(B(x,r))$ not extendible in $B(x,r)$, then the Alexandrov angle $\alpha_x(\gamma,\eta)$ between $\gamma$ and any other geodesic $\eta\in\Gamma_x(B(x,r))$ satisfies $\alpha_x(\gamma,\eta)<\pi$. In particular we would get $d_{\Sigma}(v_{\gamma},v_{\eta})<\pi$ where $v_{\gamma}, v_{\eta}\in\Sigma_x$ are the directions with respect to $\gamma$ and $\eta$, consequently $d_T((t,v_{\gamma}),(s,v_{\eta}))<t+s$ for any $t,s>0$. In such case $T_xX$ cannot be isometric to $\mathbb R^k$ for any $k=0,1,2,\cdots,d$. In this view to prove our statement it is enough to demonstrate that $T_xX$ is for $\mathscr H_d$-a.a. $x\in X$ isometric to $\mathbb R^d$.
  Let $X_k=\{x\in X\,:\,\dim_HT_xX=k\}$ for $k=0,1,2,\cdots,d$. 
   Because $X$ is separable one can cover it by countably many balls $B(x_i,r),\,i\in\mathbb N,\,r>0$. Moreover, since $X$ is locally compact and complete any ball $B(x,r)$ can be covered by a finite number of balls $B(y_i,r/2)$ where $y_i\in B[x,r]$ for $i=1,2,\cdots,N$ for some $N\in\mathbb N$ (possibly depending on $x$). In particular any ball $B(x,r)$ has bounded capacity. Denote by $\text{Reg}_kB(x,r)$ the $k$-regular part of $B(x,r)$, see \cite[Lytchak--Nagano, \S11.4]{Lytchak}, then an application of \cite[Corollary 11.8]{Lytchak} implies that $\text{Reg}_kB(x,r)$ is open in $B(x,r)$, dense in $B_k(x,r)=\{y\in B(x,r)\,:\,\dim_HT_yX=k\}$ and locally bi-Lipschitz homeomorphic to $\mathbb R^k$. Let $M_k\subseteq X_k$ be the union of $k$-regular parts of the balls $B(x_i,r)$ covering $X$. Then $M_k$ is open in $X$, dense in $X_k$ and locally bi-Lipschitz to $\mathbb R^k$. In view of \cite[Theorem 1.2]{Lytchak} the set $\overline{X}_k\setminus M_k$ is $(k-1)$--Hausdorff dimensional, by \cite[Theorem 1.3]{Lytchak} the set $M_k$ contains $ R_k=\{x\in X\,:\,T_xX\,\text{isometric to}\,\mathbb R^k\}$ and $M_k\setminus R_k$ is $(k-2)$--Hausdorff dimensional. In particular $R_d\subseteq M_d$ and it satisfies $\mathscr H_d(X\setminus R_d)=0$. 
  \end{proof}

 \subsection{Positive angles}
 \begin{mydef}
 	\label{d:normalpoint} A $\CAT(0)$ space $(X,d)$ has the positive angles at $x\in X$ if there exists $B(x,r),\,r>0$ such that for any $\gamma,\eta\in\Gamma_x^0(B(x,r))$ the Alexandrov angle $\alpha_x(\gamma,\eta)$ is positive whenever $\gamma$ and $\eta$ are not subsets of one another. If $X$ has this property at every $x\in X$, then $X$ enjoys the positive angles. If $\inf r>0$, then $X$ satisfies this property uniformly.
 \end{mydef}

\begin{proposition}
	\label{p:normal-points}Let $X$ have the positive angles at $x\in X$ for some $r>0$, then
	\begin{enumerate}[(i)]
		\item $B(x,r')$ enjoys the positive angles at $x$ for any $r'<r$;
		\item $B[x,r]$ enjoys the positive angles at $x$.
	\end{enumerate}
\end{proposition}
\begin{proof}
	\begin{enumerate}[(i)]
		\item For any $r'<r$ and $\gamma,\eta\in\Gamma^0_x(B(x,r'))$ implies $\gamma,\eta\in \Gamma^0_x(B(x,r))$. The claim follows from Definition \ref{d:normalpoint} and \eqref{eq:Alexandrov's angle}.
		\item Let $\gamma,\eta\in\Gamma^0_x(B[x,r])$ so that one is not subset of the other. Denote by $y=\gamma(1)$ and $z=\eta(1)$.
		Define $y_n=\gamma(t_n),\,z_n=\eta(t_n)$ for $n\in\mathbb N$ and a sequence $(t_n)_{n\in\mathbb N}\subset(0,1)$ satisfying $\lim_{n\to+\infty}t_n=1$. Evidently $y_n,z_n\in B(x,r)$ for every $n\in\mathbb N$ and in particular $[x,y_n],[x,z_n]\in\Gamma^0_x(B(x,r))$.
		Positive angles at $x\in X$ with $r>0$ implies $\alpha_x([x,y_n],[x,z_n])>0$ for every $n\in\mathbb N$. Note that $y_{n}\in[x,y],\,z_{n}\in[x,z]$ for every $n\in\mathbb N$, we have that $\alpha_x([x,y_n],[x,z_n])=\alpha_x([x,y_m],[x,z_m])=\alpha$ for all $m,n\in \mathbb N$ for some $\alpha>0$. Because $\lim_{n\to+\infty}d(y,y_{n})=0$ and $\lim_{n\to+\infty}d(z,z_{n})=0$, then continuity of the mapping $(y,z)\mapsto\alpha_x([x,y],[x,z])$, see e.g. \cite[Prop. 1.2.8]{Bacak} yields $$\alpha_x(\gamma,\eta)=\alpha_x([x,y],[x,z])=\lim_{n\to+\infty}\alpha_x([x,y_n],[x,z_n])=\alpha>0.$$
		Since $\gamma,\eta\in\Gamma^0_x(B[x,r])$ are arbitrary, then the result follows.
	\end{enumerate}
\end{proof}
The condition of positive angles seems quite restrictive for the class of $\CAT(0)$ spaces, but we only need this property to hold $\mathscr H_{d}$-a.e.. For practical purposes many $\CAT(0)$ spaces have positive angles almost everywhere, in particular any Riemannian manifold of non-positive curvature, any tree-like space and $\CAT(0)$ polyhedral complexes, for an application see e.g. the BHV space of phylogenetic trees \cite[Billera et al.]{BHV}. 
 Nevertheless one can construct locally compact $\CAT(0)$ spaces where this property fails anywhere. We illustrate this constructively with the following example.
 
\begin{figure}[t]
	\includegraphics[width=16cm]{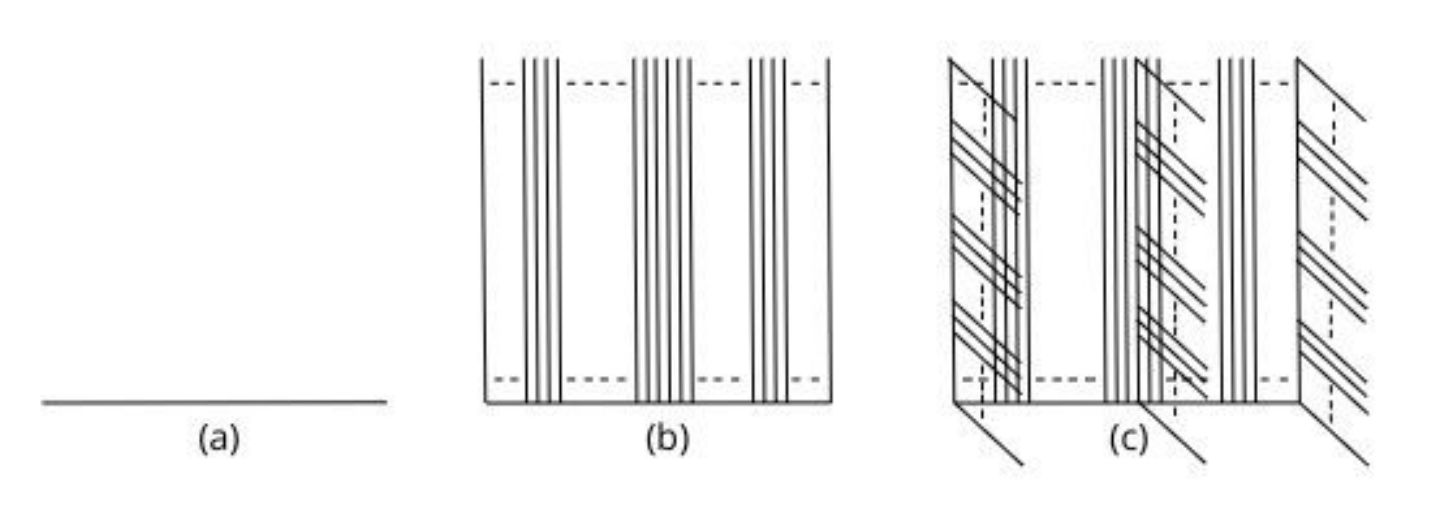}
	\caption{(a) the space $X_0$, (b) the rational comb $X_1$ and (c) a general iterated rational comb $X_k$.}
	\label{f:rational-comb}
\end{figure}
\subsubsection{The example of the iterated rational comb}
Let $X_0=[x_0,y_0]$ be a segment of unit length equipped with the length metric $d_0$. Then at each rational point on $X_0$ we glue a copy of $X_0$, this is the next iterated space $X_1$ which again is equipped with the length metric $d_1$. Note that by a rational point on $X_0$ we mean an element $p\in X_0$ such that $d_0(x_0,p)$ is a rational number. The space $X_1$ resembles a comb, see Figure \ref{f:rational-comb} (b). Then at each rational point of every tooth of $X_1$ we glue a copy of $X_0$, and this will be the next iterated space $X_2$ and we equip it with its length metric $d_2$. In this fashion we build a general iterated comb space $X_k$, see Figure \ref{f:rational-comb}(c). In particular $X_k$ is obtained from $X_{k-1}$ by gluing a copy of $X_0$ at each rational point of every last tooth of $X_{k-1}$. By virtue of Reshetnyak's Theorem, e.g. see \cite[Theorem 11.3]{Brid}, $X_k$ is a $\CAT(0)$ space for every $k\in\mathbb N_0$. For $X_0$ trivially we have $\alpha_x(\gamma,\eta)=\pi$ for all $x\in\inte X_0$ and all $\gamma,\eta\in\Gamma_x^0(X)$ such that $\gamma$ and $\eta$ are not subsets of one another, hence almost everywhere the positive angles property is satisfied. 
Consider $X_1$, we claim that at no point in the segment $X_0$ does $X_1$ satisfy the positive angles property. First to distinguish between the base segment $X_0$ of $X_1$ and other glued copies of it, we write $X^1_0(r)$ for the copy of $X_0$ glued on $X_0$ at the rational point $p(r)$ with $d_1(x_0,p(r))=r\in[0,1]$. Take some $x\in X_0$. Since rationals are dense in $[0,1]$ then for any $\varepsilon>0$ there exists a rational number $r\in[0,1]$ with $d_1(x,p(r))<\varepsilon$. Let $d_1(x,p(r))=s$ where $s\in(0,\varepsilon)$ and $y\in X^1_0(r)$ such that $d_1(y,p(r))<\varepsilon-s$, then $d_1(x,y)=d(x,p(r))+d(p(r),y)<s+(\varepsilon-s)=\varepsilon$, i.e. $y\in B(x,\varepsilon)$. Let $y'\in X_0$ with $s<d_1(x,y')<\varepsilon$. Then we have two geodesic segments $[x,y]\cap [x,y']=[x,r]$, consequently $\alpha_x([x,y],[x,y'])=\alpha_x([x,r],[x,r])=0$. Elementary calculations show that $\dim_HX_0=1$ and $\mathscr H_1(X_0)=1>0$, i.e. the set at which $X_1$ does not enjoy positive angles property has $\mathscr H_1$-measure $\geq 1$. Now define $X_{\infty}=\bigcup_{k\in\mathbb N}X_k$ equipped with the length metric. Let $x\in X_{\infty}$ then $x\in X_k$ for some $k\in\mathbb N$. Then $x$ is an element of a copy of $X_1$ lying either in $X_k$ or in $X_{k+1}$. By the previous argument for $X_1$ the space $X_{\infty}$ cannot enjoy the positive angles property at $x$, and since $x$ is arbitrary
at no point has $X_{\infty}$ the positive angles.

\subsection{Disintegration regularity}
\label{ss:disintegration}
Given a measure space $(X,\Sigma,\mu)$ 
let $(X_{t})_{t\in T}$ be a partition of $X$ and $p:x\in X\mapsto[x]\in T$ the quotient map. Equip $T$ with the largest $\sigma$-algebra $\mathscr T$ for which $p$ is measurable, i.e. $p^{-1}(F)\in\Sigma$ whenever $F\in\mathscr T$. Suppose that $\mu$ has total finite variation, meaning $|\mu|(X)=\mu^+(X)+\mu^-(X)<+\infty$ where $\mu^+$ and $\mu^-$ are the non-negative and non-positive part respectively of a general (possibly signed) measure $\mu$. Denote by $\nu=p_{\#}\mu$. Following \cite{Caravenna} we present the definition of disintegration of a measure.
\begin{mydef}
\label{d:disintegration-measure}
A family of probability measures $(\mu_{t})_{t\in T}$ in $X$ is a disintegration of the measure $\mu$ consistent with the partition $(X_t)_{t\in T}$ if 
\begin{enumerate}[(i)]
	\item for every $B\in\Sigma$ the mappings $t\mapsto\mu_t(B)$ is $\nu$-measurable;
	\item for every $B\in \Sigma$ and $F\in \mathscr T$ it holds 
	\begin{equation}
	\label{eq:disintegration}
	\mu(B\cap p^{-1}(F))=\int_F\mu_t(B)\,d\nu(t).
	\end{equation}
\end{enumerate}
The disintegration is unique if the measures $\mu_t$ are uniquely determined for $t\in T,\;\nu$-a.e..
\end{mydef}
For a deeper treatment on disintegration of measures see e.g. \cite[Bogachev \S 10.4, \S 10.6]{Bogachev1}. In this section we suppose for simplicity that $(X,d)$ is a bounded $\CAT(0)$ space with $\dim_HX<+\infty$, else in view of Proposition \ref{p:finite-dim} we could always restrict the analysis on closed balls. Moreover suppose that $X$ satisfies the positive angles property. 
Given $x\in X$ define $ M_x(X)$ to be the set of all geodesics $\gamma\in\Gamma^0_x(X)$ for which there is no other geodesic $\widetilde{\gamma}\in\Gamma^0_x(X)$ that properly contains $\gamma$. We refer to $ M_x(X)$ as the set of all maximal geodesics issuing from $x$. For simplicity suppose that the radius $r_x>0$ for which $X$ enjoys the positive angles property at $x$ satisfies $r_x=\sup\{d(x,y)\,:\,y\in X\}$.  
For $\gamma\in M_x(X)$ let $\widehat\gamma=\gamma\setminus\{x\}$, then the collection of sets $(\widehat\gamma)_{\gamma\in M_x(X)}\cup\{x\}$ is a partition of $X$. This is evident for if $\widehat\gamma_1\cap\widehat\gamma_2\neq\emptyset$ then there is $y\in X$ with $(x,y]\subseteq\widehat\gamma_1$ and $(x,y]\subseteq\widehat\gamma_2$ implying a zero Alexandrov angle between $\gamma_1$ and $\gamma_2$, which is impossible by positive angles at $x$ or else $\gamma_1= \gamma_2$. Since $(X,d)$ is a complete separable metric space and $\mathscr H_d$ is defined on the Borel $\sigma$-algebra $\mathscr B(X)$, then $(X,\mathscr B(X),\mathscr H_d)$ is a Lebesgue measure space, see e.g. \cite[Bogachev \S 9.4]{Bogachev1}. Moreover the measure spaces $(X,\mathscr B(X),\mathscr H_d)$ and $(X\setminus\{x\},\mathscr B(X),\mathscr H_d)$ are isomorphic$\mod 0$, under the inclusion map $\iota:X\setminus\{x\}\xhookrightarrow{} X$ and the identification of the $\mathscr H_d$-null sets $\{\emptyset\}\subset X\setminus\{x\}$ and $\{x\}\subset X$. Similarly the measure spaces $((\widehat\gamma)_{\gamma\in M_x(X)}\cup\{x\}, \mathscr T, \nu)$ and $((\widehat\gamma)_{\gamma\in M_x(X)},\mathscr T, \nu)$ with $\nu=p_{\#}\mathscr H_d$ are isomorphic$\mod 0$. Consequently it suffices to restrict to the collection $(\widehat\gamma)_{\gamma\in M_x(X)}$.
On the other hand separability implies that $(X,\mathscr B(X))$ is countably generated. Together with the condition $\mathscr H_d(X)<+\infty$ (since $X$ is bounded), we obtain by virtue of \cite[Theorem 2.3]{Caravenna}, that there exists a unique disintegration $(\mu_{\widehat\gamma})_{\gamma\in M_x(X)}$ of $\mathscr H_d$ consistent with the partition $(\widehat\gamma)_{\gamma\in M_x(X)}$. By Beppo Levi's theorem, see e.g. \cite[Remark 2.2]{Caravenna}, formula   
\eqref{eq:disintegration} is extended for measurable functions $\phi$
as the integral 
\begin{equation}
\label{eq:Beppo}
\int_X\phi(u)\,d\mathscr H_d(u)=\int_T\int_{\widehat\gamma}\phi(u)\,d\mu_{\widehat\gamma}(u)\,d\nu(\widehat\gamma).
\end{equation}
 Here $T$ is the quotient space $X/\sim$ under the equivalence relation $y\sim z$ if and only if $y,z$ are on the same $\gamma\in M_x(X)$.
\begin{mydef}
	\label{d:radial}
	The radial projection of $x\in X$ onto $\gamma\in\Gamma(X)$ is the element $R_{\gamma}(x)\in\gamma$ such that
	\begin{equation}
	\label{eq:radial-proj}
	 d(\gamma(0),R_{\gamma}(x))=\left\{
	\begin{array}{ll}
d(\gamma(0),x) & d(\gamma(0),x)< \ell(\gamma) \\[1em]
d(\gamma(0),\gamma(1)) &  d(\gamma(0),x)\geq\ell(\gamma). 
	\end{array} 
	\right. 
	\end{equation}
	 For $y\in\inte\gamma$ the sets $R^{-1}_{\gamma}(y)$ are geodesic spheres of radii $r_y=d(\gamma(0),y)$ centered at $\gamma(0)$.
\end{mydef}
 Note that $R_{\gamma}(x)$ exists and is unique for every $x$.  
 Moreover $R_{\gamma}$ is non-expansive since $$d(R_{\gamma}(x),R_{\gamma}(y))=|d(\gamma(0),R_{\gamma}(x))-d(\gamma(0),R_{\gamma}(y))|\leq|d(\gamma(0),x)-d(\gamma(0),y)|\leq d(x,y)$$ for every $x,y\in X$.
An application of (generalized) Eilenberg's inequality with $\dim_HX=d$ and $\dim_H\gamma=1$, e.g. see \cite[Theorem 1.1]{Esmayli} or \cite[Theorem 3.26]{Piotr}, implies 
\begin{equation}
\label{eq:Eilenberg}
\int_{\gamma}\int_{R_{\gamma}^{-1}(y)}\chi_S(u)\,d\mathscr H_{d-1}(u)\,d\mathscr H_1(y)\leq \mathscr H_d(S),\quad S-\text{measurable}.
\end{equation} 
The left side of \eqref{eq:Eilenberg} defines a measure on the measurable sets of $X$ which is absolutely continuous w.r.t $\mathscr H_d$. Then there exists a measurable function $\zeta_{\gamma}:X\to[0,+\infty]$ such that 
\begin{equation}
\label{eq:absolute}
\int_{\gamma}\int_{R_{\gamma}^{-1}(y)}\phi(u)\,d\mathscr H_{d-1}(u)\,d\mathscr H_1(y)=\int_X\phi(u)\,\zeta_{\gamma}(u)\,d\mathscr H_{d}(u)
\end{equation}
for all measurable functions $\phi:X\to[0,+\infty]$. 
\begin{mydef}
	\label{d:disintegration} A $\CAT(0)$ space $(X,d)$ satisfies the disintegration regularity condition at $x\in X$ if there exists $r>0$ and $\gamma\in\Gamma_x^0(B[x,r])$ such that 
		\begin{enumerate}[(i)]
			\item for every measurable function $\phi:X\to[0,+\infty]$
			\begin{equation}
			\label{eq:absoluteS}
			\int_{\gamma}\int_{R_{\gamma}^{-1}(y)\cap B[x,r]}\phi(u)\,d\mathscr H_{d_B-1}(u)\,d\mathscr H_1(y)=\int_{B[x,r]}\phi(u)\,\zeta_{\gamma}(u)\,d\mathscr H_{d_B}(u)
			\end{equation}
		where $\zeta_{\gamma}>0,\;\mathscr H_{d_B}-$a.e. on $B=B[x,r]$; 
			\item there is a family of maps $\Psi_y:R_{\gamma}^{-1}(\gamma(1))\cap B[x,r]\to R_{\gamma}^{-1}(y)\cap B[x,r],\,y\in\gamma$ such that $$(\Psi_y^{-1})_{\#}(\mathscr H_{{d_B}-1}|_{\Psi_y(R_{\gamma}^{-1}(\gamma(1)))\cap B[x,r]})\ll \mathscr H_{{d_B}-1}|_{R_{\gamma}^{-1}(\gamma(1))\cap B[x,r]}\;.$$
		\end{enumerate}
	If $X$ satisfies this property at every $x\in X$ then we simply say that $(X,d)$ enjoys the disintegration regularity condition. If $\inf r>0$ then $(X,d)$ satisfies this property uniformly.
\end{mydef}

\begin{lemma}
	\label{l:mappings} Let $x\in X$ and consider the ball $B[x,r]$ for $r>0$.
	A family of mappings $\Psi_y:R_{\gamma}^{-1}(\gamma(1))\cap B[x,r]\to R_{\gamma}^{-1}(y)\cap B[x,r],\,y\in\gamma$ satisfying condition (ii) in Definition \ref{d:disintegration} is given by $\Psi_y(v)=(1-t(y))\,\gamma(0)\oplus t(y)\,v$ where $\gamma(t(y))=y$.
\end{lemma}
\begin{proof}
	First note that by joint convexity it follows $d(\Psi_y(u),\Psi_y(v))\leq t(y)\,d(u,v)$ for every $u,v\in R_{\gamma}^{-1}(\gamma(1))$. In particular $\Psi_y$ is nonexpansive for every $y\in\gamma$. Let $A\in\mathscr B(B[x,r])\cap R_{\gamma}^{-1}(\gamma(1))$ with $\mathscr H_{{d_B}-1}(A)=0$. For $\delta>0$ consider a countable cover $(U_k)_{k\in\mathbb N}$ of $A$ with $\diam U_k<\delta$ for all $k\in\mathbb N$. Then $(\Psi_y(U_k))_{k\in\mathbb N}$ is a countable cover for the set $\Psi_y(A)$, since $\Psi_y(A)\subseteq\Psi_y(\bigcup_{k\in\mathbb N}U_k)=\bigcup_{k\in\mathbb N}\Psi_y(U_k)$. Moreover 
	\begin{align*}
	\diam \Psi_y(U_k)&=\sup\{d(\Psi_y(x_1), \Psi_y(x_2))\,:\,x_1,x_2\in U_k\}\\&\leq \sup\{d(x_1,x_2)\,:\,x_1,x_2\in U_k\}=\diam U_k,\quad\forall k\in\mathbb N.
	\end{align*}
	In particular $\diam \Psi_y(U_k)<\delta$ for every $k\in\mathbb N$. Then $H^{d_B-1}_{\delta}(\Psi_y(A))\leq H^{d_B-1}_{\delta}(A)$ for every $\delta>0$ implies $\mathscr H_{d_B-1}(\Psi_y(A))\leq\mathscr H_{d_B-1}(A)$, consequently $\mathscr H_{d_B-1}(\Psi_y(A))=0$.

\end{proof}

\begin{thm}
	\label{th:H1}
	If $(X,d)$ satisfies the disintegration regularity, then for every $x\in X$ there exists a closed ball $B=B[x,r],\,r>0$ such that $\mathscr H_{d_B}$ admits a unique disintegration  consistent with the partition $(\widehat\gamma)_{\gamma\in M_x(B)}$ that is absolutely continuous with respect to $\mathscr H_1$.
\end{thm}
\begin{proof}
Disintegration regularity implies disintegration regularity at every $x\in X$. 
Let $B[x,r]$ be the closed ball at $x$ with radius $r>0$ (possibly dependent on $x$) for which the conditions (i)(ii) in Definition \ref{d:disintegration} are fulfilled. Let $d_B:=\dim_HB[x,r]$.
Since $X$ is locally compact then $B$ is compact. In particular it is separable, thus $(B[x,r],\mathscr B(B[x,r]),\mathscr H_{d_B})$ is countably generated. Moreover $\mathscr H_{d_B}(B)<+\infty$. 
There is a unique disintegration $(\mu_{\widehat\gamma})_{\gamma\in M_x(B)}$ of $\mathscr H_{d_B}$ consistent with the partition $(\widehat\gamma)_{\gamma\in M_x(B)}$. For simplicity assume elements in $ M_x(B)$ have equal length.
	From condition (i) it follows that for some $\zeta_{\gamma}>0,\;\mathscr H_{d_B}$-a.e. on $B[x,r]$ we have
	\begin{equation}
	\label{eq:zeta}
	\int_{B[x,r]}\phi(u)\,\zeta_{\gamma}(u)\,d\mathscr H_{d_B}(u)=\int_{\gamma}\int_{R_{\gamma}^{-1}(y)\cap B[x,r]}\phi(u)\,d\mathscr H_{d_B-1}(u)\,d\mathscr H_1(y),\quad\phi-\text{measurable}.
	\end{equation}
	Substituting $\Psi_{y}(v)=u$ for $v\in R_{\gamma_B}^{-1}(\gamma(1))\cap B[x,r]$, then by disintegration regularity condition (ii) we obtain 
	\begin{align*}
\label{eq:lemma}
\int_{\gamma}\int_{R_{\gamma}^{-1}(\gamma(1))\cap B[x,r]}\phi(\Psi_{y}(v))&\,d\mathscr H_{d_B-1}(\Psi_{y}(v))\,d\mathscr H_1(y)
\\&=\int_{\gamma}\int_{R_{\gamma}^{-1}(\gamma(1))\cap B[x,r]}\phi(\Psi_{y}(v))\,\theta(v,y)\,d\mathscr H_{d_B-1}(v)\,d\mathscr H_1(y)
	\end{align*}
	for a certain non-negative measurable function $\theta(v,y)$. Applying Fubini--Tonelli Theorem in the last double integral together with \eqref{eq:zeta} we get that
	\begin{align*}
		\int_{B[x,r]}\phi(u)\,\zeta_{\gamma}(u)\,d\mathscr H_{d_B}(u)=\int_{R_{\gamma}^{-1}(\gamma(1))\cap B[x,r]}\int_{\gamma}\phi(\Psi_{y}(v))\,\theta(v,y)\,d\mathscr H_1(y)\,d\mathscr H_{d_B-1}(v).
	\end{align*}
	Under the bijection $\phi\mapsto \phi/\zeta_{\gamma}$, and this is justified since $\zeta_{\gamma}>0-\mathscr H_{d_B}$-a.e. on $B[x,r]$,  
	the last equation can be written as 
	\begin{align}	
	\int_{B[x,r]}\phi(u)\,d\mathscr H_{d_B}(u)=\int_{R_{\gamma}^{-1}(\gamma(1))\cap B[x,r]}\int_{\gamma}\frac{\phi(\Psi_{y}(v))}{\zeta_{\gamma}(\Psi_{y}(v))}\,\theta(v,y)\,d\mathscr H_1(y)\,d\mathscr H_{d_B-1}(v).
	\end{align}
Consider the function 
	$$\vartheta(v)=\int_{\gamma}\frac{\theta(v,y)}{\zeta_{\gamma}(\Psi_{y}(v))}\,d\mathscr H_1(y),\quad v\in R_{\gamma_B}^{-1}(\gamma_B(1))\cap B[x,r].$$
	Note that each element $v\in R_{\gamma}^{-1}(\gamma(1))\cap B[x,r]$ can be identified with a unique element $[\widehat\gamma_v]\in T$ under the quotient mapping $p$, in particular we have $v=\gamma_v(1)$.
	By change of variables $y=\gamma(t)$ and choosing $\Psi_{y}$ as in Lemma \ref{l:mappings} then yields
	$$\vartheta(\gamma_v(1))=\frac{\ell(\gamma)}{\ell(\gamma_v)}\cdot\int_{0}^1\frac{\theta(\gamma_v(1),\gamma(t))}{\zeta_{\gamma}(\gamma_v(t))}\,d\mathscr H_1(\gamma_v(t)),\quad  v\in  R_{\gamma}^{-1}(\gamma(1))\cap B[x,r].$$
	We define the family of probability measures
	\begin{equation*}
	d\widetilde\mu_{\widehat\gamma_v}(t)=\frac{d\vartheta(\gamma_v(t))}{\vartheta(\gamma_v(1))},\quad\text{where}\;\vartheta(\gamma_v(t))=\frac{\ell(\gamma)}{\ell(\gamma_v)}\cdot\int_{0}^t\frac{\theta(\gamma_v(1),\gamma(s))}{\zeta_{\gamma}(\gamma_v(s))}\,d\mathscr H_1(\gamma_v(s))
	\end{equation*}
	for all $t\in[0,1]$ and $v\in  R_{\gamma}^{-1}(\gamma(1))\cap B[x,r]$.
	By construction $\widetilde\mu_{\widehat\gamma_v}\ll\mathscr H_1|_{\widehat\gamma}$ for every $v\in  R_{\gamma}^{-1}(\gamma(1))\cap B[x,r]$. 
	By uniqueness of disintegration $\mu_{\widehat\gamma}=\widetilde\mu_{\widehat\gamma}$ for $\nu$-a.a. $[\widehat\gamma]\in T$.
\end{proof}

\begin{question}
Does every locally compact and complete $\CAT(0)$ space satisfy the disintegration regularity? 
\end{question}

\subsubsection{Hausdorff dimension revisited} In this section we present a theorem that establishes a relationship between the positive angles property and the Hausdorff dimensions of a set $S\subseteq X$ and of $\partial S(x,t)=\{\gamma(t)\,:\,\gamma\in M_x(S)\}$, where $x\in S$ and $t\in[0,1]$. 
\begin{thm}
	\label{th:boundary}
For any bounded connected set $S\subseteq X$
the following inequality holds $\dim_H\partial S(x,t)\leq\dim_HS-1$ for every $t\in[0,1)$. If additionally $S$ satisfies the positive angles property at every $x\in S$ with $r_x\geq \diam S$ then the result holds true also for $t=1$.
\end{thm}
\begin{proof}
	Denote by $d=\dim_HS, s(t)=\dim_H\partial S(x,t)$ and let $s=s(1)$.
We define a family of mappings $\psi_t:\partial S(x,1)\to\partial S(x,t)$ with $\psi_t(y)=(1-t)x\oplus ty$ for $y\in\partial S(x,1)$. 
Then $\psi_t$ is Lipschitz and consequently $s(t)\leq  s$ for every $t\in(0,1]$. Moreover $s(t)$ is non-decreasing in $t$. Now pick any $\gamma\in M_x(S)$ and set $\varphi_t:\partial S(x,t)\to\gamma$ by the rule $\varphi_t(y)=\gamma(t)$ for $y\in\partial S(x,t)$. Note that $d(\varphi_t(y_1),\varphi_t(y_2))=0$ for any $y_1,y_2\in\partial S(x,t)$, in particular $\varphi_t$ is Lipschitz. 
By Eilenberg's inequality we obtain 
	$$\int_{\gamma}\int_{\varphi_t^{-1}(y)\cap S}\,d\mathscr H_{d-1}(u)\,d\mathscr H_1(y)\leq 0.$$
	Since the left side is non-negative then the integral vanishes identically, implying that $\mathscr H_{d-1}(\varphi_t^{-1}(y)\cap S)=\mathscr H_{d-1}(\varphi_t^{-1}(y))=0$ for $\mathscr H_1$-a.a. $y\in\gamma$. Realizing that $\varphi_t^{-1}(y)=\partial S(x,t)$ this is equivalent to $\mathscr H_{d-1}(\partial S(x,t))=0$ for almost all $t\in[0,1]$, which in turn implies that $s(t)\leq d-1$ for almost all $t\in[0,1]$. However we can do slightly better then this. 
	If $s(t^*)>d-1$ for some $t^*\in[0,1)$ then $s(t)>d-1$ for all $t\geq t^*$. Again the relation $\varphi_t^{-1}(y)=\partial S(x,t)$, yields $$\int_{\varphi_{t}^{-1}(y)\cap S}\,d\mathscr H_{d-1}(u)=\int_{\partial S(x,t)\cap S}\,d\mathscr H_{d-1}(u)=\int_{\partial S(x,t)}\,d\mathscr H_{d-1}(u)=\mathscr H_{d-1}(\partial S(x,t))=+\infty$$
	for all $t\in[t^*,1]$, consequently 
	$$\int_{\gamma}\int_{\varphi_t^{-1}(y)\cap S}\,d\mathscr H_{d-1}(u)\,d\mathscr H_1(y)\geq \int_{\gamma|_{[t^*,1]}}\int_{\varphi_{t}^{-1}(y)\cap S}\,d\mathscr H_{d-1}(u)\,d\mathscr H_1(y)=+\infty$$
	 which is impossible. Therefore $\dim_H(\partial S(x,t))\leq d-1$ for all $t\in[0,1)$.
	 Now let $S$ satisfy the positive angles property at every $x\in S$ with $r_x\geq \diam S$. Let $\delta>0$ be sufficiently small and $t= 1-\delta$, then there is $\varepsilon>0$ such that $d(y,z)\leq (1+\varepsilon)\,d(\psi_t(y),\psi_t(z))$ for all $y, z\in\partial S(x,1)$. If this was not the case then for every $n\in\mathbb N$ there would exist $y_{n}, z_{n}\in\partial S(x,1)$ such that $d(y_n,z_n)>(1+n)\,d(\psi_{1-\frac{1}{n}}(y_n),\psi_{1-\frac{1}{n}}(z_n))$. Clearly $y_n\neq z_n$ for all $n\in\mathbb N$. We then obtain 
	 \begin{align*}
	 +\infty>\diam \partial S(x,1)\geq \sup_{n\in\mathbb N}d(y_n,z_n)&\geq \sup_{n\in\mathbb N}(1+n)\,d(\psi_{1-\frac{1}{n}}(y_n),\psi_{1-\frac{1}{n}}(z_n))\\&
	 \geq (1+N)\sup_{n\geq N}d(\psi_{1-\frac{1}{N}}(y_n),\psi_{1-\frac{1}{N}}(z_n)), \quad N\in \mathbb N.
	 \end{align*}
	The last inequality follows from the non-positive curvature of the space. Positive angles property then implies $\sup_{n\geq N}d(\psi_{1-\frac{1}{N}}(y_n),\psi_{1-\frac{1}{N}}(z_n))>0$. But then letting $N\to+\infty$ would raise a contradiction. Therefore for sufficiently small $\delta>0$ and $t=1-\delta$, the mapping $\psi_t$ has a Lipschitz continuous inverse. Then by Lemma \ref{l:Hausdorff} (vii) and first part of the statement we have $\dim_H\partial S(x,1)\leq \partial S(x,t)\leq d-1$. 
\end{proof}

\section{Two basic lemmas}
\label{s:basic-lemmas}

\begin{lemma}
	\label{l:min-max} Let $f:X\to(-\infty,+\infty]$ be geodesically differentiable. If $x^*$ is a minimizer then $Df(x^*;\gamma)\geq 0$ for all $\gamma\in \Gamma^0_{x^*}(X)$. In particular $Df(x^*;\gamma)= 0$ for all $\gamma\in\Gamma_{x^*}(X)$ whenever $X$ enjoys the local geodesic extension property.
\end{lemma}
\begin{proof}
	Let $x^*$ be a minimizer of $f$. Then for all $x\in X$ it holds that $f(x)\geq f(x^*)$. In particular for any $\gamma\in\Gamma_{x^*}(X)$ we have that $f(\gamma(h))\geq f(x^*)$ for all $h\in[0,1]$. First let $\gamma(0)=x^*$, then
	$$Df(x^*;\gamma)=\lim_{h\to 0}\frac{f(\gamma(h))-f(\gamma(0))}{h}\geq 0.$$
	This proves the first claim. Next suppose that $(X,d)$ enjoys the local geodesic extension property. Then for every $\gamma\in\Gamma_{x^*}(X)$ there exists $\widetilde\gamma\in\Gamma_{x^*}(X)$ with $\gamma\subset\inte\widetilde\gamma$. Using the definition of the geodesic derivative one obtains the relation $$Df(x^*;\gamma)=\frac{\ell(\gamma)}{\ell(\widetilde\gamma)}\,Df(x^*;\widetilde\gamma).$$
	Now assume w.l.o.g. that $x^*\in \gamma$ is an interior point in $\gamma$, i.e. $x^*=\gamma(s)$ for some $s\in(0,1)$. Then from the two inequalities 
	\begin{align*}
	&\frac{f(\gamma(s+h))-f(\gamma(s))}{h}\geq 0,\quad h>0\quad\text{and}\quad\frac{f(\gamma(s+h))-f(\gamma(s))}{h}\leq 0,\quad h<0
	\end{align*}
	one obtains $0\geq D^-f(x^*;\gamma)= D^+f(x^*;\gamma)\geq 0$ and subsequently $Df(x^*;\gamma)=0$.
\end{proof}

\begin{lemma}
	\label{l:twist}
	Let $(X,d)$ be a $\CAT(0)$ space enjoying the positive angles property. Then for every $x\in X$ there exists $r>0$ such that $D_xc(x,y_1;\gamma)=D_xc(x,y_2;\gamma)$ for all $\gamma\in\Gamma_x(X)$ implies $y_1=y_2$ whenever $y_1,y_2\in B(x,r)$.
\end{lemma}
\begin{proof} 
	Let $(X,d)$ satisfy the positive angles property and take $x\in X$. By Definition \ref{d:normalpoint} there exists $B(x,r),\,r>0$ such that for any $\gamma_1,\gamma_2\in\Gamma_x^0(B(x,r))$ with $\gamma_1$ and $\gamma_2$ not subsets of one another, the angle $\alpha_x(\gamma_1,\gamma_2)$ is strictly positive. 
	Denote by $\gamma_i(1)=y_i,\,i=1,2$. First note that by Definition \ref{d:geo-derivative} for all $\gamma\in\Gamma(X)$ and $u, v\in\gamma$ that
	\begin{equation}
	\label{eq:geo-der}
	D_uc(u,v;\gamma)=(t-s)\,d(\gamma(0),\gamma(1))^2
	\end{equation} 
	where $u=\gamma(t), v=\gamma(s)$ for some $s,y\in[0,1]$. From \eqref{eq:geo-der} follows $D_xc(x,y_i;\gamma_i)=-d(x,y_i)^2$ for $i=1,2$. By geodesic derivative and the reverse triangle inequality for $i,j=1,2$ 
	\begin{align*}
	D_xc(x,y_i;\gamma_j)&=\frac{1}{2}\,\lim_{h\to 0}\frac{d(\gamma_j(h),y_i)^2-d(x,y_i)^2}{h}\\&
	\geq\frac{1}{2}\, \lim_{h\to 0}\frac{(d(x,y_i)-d(x,\gamma_j(h)))^2-d(x,y_i)^2}{h}\\&
	=\frac{1}{2}\,\lim_{h\to 0}\frac{d(x,\gamma_j(h))^2-2d(x,\gamma_j(h))\,d(x,y_i)}{h}\\&
	=\frac{1}{2}\,\lim_{h\to 0}\frac{h^2\,d(x,y_j)^2-2h\,d(x,y_j)\,d(x,y_i)}{h}=-\,d(x,y_j)\,d(x,y_i).
	\end{align*}
	Together with $D_xc(x,y_i;\gamma_i)=D_xc(x,y_j;\gamma_i)$, it then implies $d(x,y_i)=d(x,y_j)$, i.e. $y_1,y_2$ must be equidistant from $x$. 
	Let $P_{\gamma_i}(\gamma_j(h))=\gamma_i(t(h))$ where $t:[0,1]\to[0,1]$ then
	\begin{align*}
	D_xc(x,y_i;\gamma_j)&=\frac{1}{2}\,\lim_{h\to 0}\frac{d(\gamma_j(h),y_i)^2-d(x,y_i)^2}{h}\\&\geq \frac{1}{2}\,\lim_{h\to 0}\frac{d(\gamma_j(h),\gamma_i(t(h)))^2+d(\gamma_i(t(h)),y_i)^2-d(x,y_i)^2}{h}\\&\geq \frac{1}{2}\,\lim_{h\to 0}\frac{d(\gamma_i(t(h)),y_i)^2-d(x,y_i)^2}{h}
	=\frac{1}{2}\,d(x,y_i)^2\lim_{h\to 0}\frac{-t(h)(2-t(h))}{h}.
	\end{align*}
	First notice that if $P_{\gamma_i}(\gamma_j(h))=x$, i.e. $\gamma_i(t(h))=x$, then $D_xc(x,y_i;\gamma_j)\geq 0$ and in particular by hypothesis $D_x(x,y_j;\gamma_j)\geq 0$, which together with $D_xc(x,y_j;\gamma_j)=-d(x,y_j)^2\leq 0$ implies that $x=y_j$ and consequently $x=y_i$, hence $y_1=y_2$. 
	For a given $h>0$ let $\Delta(\overline \gamma_j(h),\overline x,\overline \gamma_i(t(h)))$ be the comparison triangle in $\mathbb R^2$ for $\Delta( \gamma_j(h),x, \gamma_i(t(h)))$. By non-positive curvature of the space we have that $\alpha_x(\gamma_j(h),\gamma_i(t(h)))\leq \overline\alpha_{\overline x}(\overline\gamma_j(h),\overline\gamma_i(t(h)))$ for $h>0$. Let $\ell$ be the line extending from the segment $[\overline x, \overline\gamma_i(t(h))]$ in $\mathbb R^2$ and let $P_{\ell}(\overline \gamma_j(h))$ denote the projection of $\overline \gamma_j(h)$ onto $\ell$. Note that by Lemma \ref{l:projections} the comparison angle at $\gamma_i(t(h))$ satisfies $$\overline{\alpha}_{\overline{\gamma}_i(t(h))}([\overline{\gamma}_i(t(h)),\overline{x}],[\overline{\gamma}_i(t(h)),\overline{\gamma}_j(h)])\geq \alpha_{\gamma_i(t(h))}([\gamma_i(t(h)),x],[\gamma_i(t(h)),\gamma_j(h)])\geq\frac{\pi}{2}.$$ Elementary calculations about Euclidean triangles yield the inequality
	$$\frac{t(h)}{h}\leq \frac{|\overline x\,P_{\ell}(\overline \gamma_j(h))|}{|\overline x\,\overline\gamma_j(h)|}=\cos\overline \alpha_{\overline x}(\overline\gamma_j(h),\overline\gamma_i(t(h))).$$
	Passing in the upper limit as $h\to 0$ gives
	\begin{align*}
	0\leq\limsup_{h\to 0}\frac{t(h)}{h}\leq \limsup_{h\to 0}\cos\overline \alpha_{\overline x}(\overline\gamma_j(h),\overline\gamma_i(t(h)))=\cos \alpha_x(\gamma_j,\gamma_i)<1
	\end{align*}
	where the last inequality follows for all $x\in X$. Therefore we get $D_xc(x,y_i;\gamma_j)>-\,d(x,y_i)^2$. On the other hand $D_xc(x,y_i;\gamma_j)=D_xc(x,y_j;\gamma_j)=-\,d(x,y_j)^2$ implies $d(x,y_j)<d(x,y_i)$, which raises a contradiction as $y_1,y_2$ are equidistant from $x$.
\end{proof}

\section{Existence and uniqueness of optimal transport maps}
\label{s:main}
\subsection{Optimal transport maps in $\CAT(0)$ spaces}
\begin{thm}
	\label{th:1} Let $(X,d)$ with $\dim_HX<+\infty$ be a locally compact complete $\CAT(0)$ space satisfying the following properties $\mathscr H_{d}$-a.e. uniformly:
	\begin{enumerate}[(i)]
		\item positive angles;
		\item the disintegration regularity.
	\end{enumerate}
	  Let $\mu,\nu\in\mathscr P_2(X)$ with $\mu$ absolutely continuous.  Then there exists a unique optimal transport plan $\pi\in\Pi(\mu,\nu)$. Moreover $\pi=(\Id,T)_{\#}\mu$ for a Borel measurable mapping $T:X\to X$ such that $T(x)\in\partial_c\psi(x)$ where $\psi$ is a $c$-convex function and satisfies
	\begin{equation}
	\label{eq:T}
	\widetilde D\psi(x;\gamma)+D_xc(x,T(x);\gamma)=0,\quad\forall \gamma\in\Gamma_x(X)\; \mu-a.e..
	\end{equation} 
\end{thm}
\begin{proof}
By Remark \ref{r:c-convex} there exists an optimal transport plan $\pi\in\Pi(\mu,\nu)$. Lemma \ref{l:Kantorovich} implies that $\pi$-almost surely $\psi^c(y)-\psi(x)=c(x,y)$ for some $c$-convex function $\psi$. 
Evidently any such $y$ satisfies $y\in\partial_c\psi(x)$. Next we show that this $y$ is unique. It suffices to prove that $T(x)=y$ for a certain Borel measurable mapping $T:X\to X$.  
	
For a given $y_0\in X$ and $R>0$ define the function $\psi_R(x)=\inf_{y\in B(y_0,R)}(\psi^c(y)-c(x,y))$. We claim that $\psi_R$ is locally Lipschitz in $X$ for sufficiently large $R$ (enough that $\psi^c(y)<+\infty$ for some $y\in B(y_0,R)$).
Indeed let $x_0\in X$ and $B[x_0,r]$ be a closed geodesic ball centered at $x_0$ of radius $r>0$. For $x_1,x_2\in B[x_0,r]$ it holds
	\begin{align*}
	|\psi_R(x_1)-\psi_R(x_2)|&=|\inf_{y\in B(y_0,R)}(\psi^c(y)-c(x_1,y))-\inf_{y\in B(y_0,R)}(\psi^c(y)-c(x_2,y))|\\&
	\leq \sup_{y\in B(y_0,R)}|c(x_2,y)-c(x_1,y)^2|\leq 2r\,d(x_1,x_2).
	\end{align*}
Let $x\in X$ and $\gamma\in \Gamma_x(X)$. 
	Denote by $$\Omega=\{x\in X\,:\,D^+\psi_R(x;\gamma)>D^-\psi_R(x;\gamma)\,\text{for some geodesic}\,\gamma\in\Gamma_x(X)\}.$$ Define $\varphi(t;x,\gamma):=\psi_R(\gamma(t))$, for any $s,t\in[0,1]$ it follows $|\varphi(s;x,\gamma)-\varphi(t;x,\gamma)|\leq 2r\ell(\gamma)\,|s-t|$ for some $r>0$. In particular $\varphi(\cdot\,;x,\gamma)$ is Lipschitz continuous on $[0,1]$. By Lebesgue's differentiation theorem then $\varphi(\cdot\,;x,\gamma)$ is differentiable almost everywhere on $[0,1]$ and consequently $\psi_R(y)$ is $\mathscr H_1$-a.e. differentiable on $\gamma$.   
 Let $\{x_k\}_{k\in\mathbb N}\subseteq X$ such that the positive angles property and disintegration regularity are satisfied at every $x_k$. Let $B[x_k,r]$ be the balls where these two properties hold in view of Definitions \ref{d:normalpoint} and \ref{d:disintegration} (w.l.o.g. $r>0$ is the same for both conditions). Moreover note by Proposition \ref{p:normal-points} that positive angles property extends to closed ball $B[x,r]$, whenever it holds in its interior. 
 By Lemma \ref{l:sigma-finite} there exists a countable dense set $\{y_k\}_{k\in\mathbb N}\subseteq X$ such that  $X=\bigcup_{k\in\mathbb N}B(y_k,r/2)$. Let $\dim_HB(y_k,r/2)=d_k$.
 We can pick $x_k$ from $B(y_k,r/2)$ whenever $d_k=d(=\dim_HX)$. Then we can write 
 $$X=\Big(\bigcup_{k\in\mathbb N, d_k<d}B(y_k,r/2)\Big)\cup\Big(\bigcup_{k\in\mathbb N}B[x_k,r]\Big)$$
 since $B(y_k,r/2)\subset B[x_k,r]$ whenever $x_k$ is chosen from $B(y_k,r/2)$ with $d_k=d$. Moreover note that $\dim_HB[x_k,r]=d$ for all $k\in\mathbb N$.
  By Theorem \ref{th:H1} the Hausdorff measure $\mathscr H_{d}$ admits a disintegration $(\mu^k_{\widehat\gamma})_{\widehat\gamma\in M_{x_k}(X_k)}$, where $X_k=B[x_k,r]$, that is absolutely continuous with respect to $\mathscr H_1$. For each $k\in\mathbb N$ we have 
 $$\int_{\Omega\cap X_k}\,d\mathscr H_d(y)=\int_{T_k}\mu^k_{\widehat\gamma}(p_k^{-1}([\widehat\gamma])\cap \Omega)\,d\mathscr H_{d}(p^{-1}_k([\widehat\gamma])),\quad T_k=X_k/\sim.$$
 Since $\psi_R(\gamma(t))$ is differentiable a.e. on $[0,1]$ for every $\gamma\in\Gamma(X)$ then in particular $\psi_R(\widehat\gamma(t))$ is so for every $\widehat\gamma\in M_{x_k}(X_k)$, consequently $D^+\psi_R(y;\widehat\gamma)=D^-\psi_R(y;\widehat\gamma),\;\mathscr H_1-$a.e. on $\widehat\gamma$. Therefore $\mu^k_{\widehat\gamma}(p_k^{-1}([\widehat\gamma])\cap \Omega)=0$ for all $\widehat\gamma\in M_{x_k}(X_k)$, for every $k\in\mathbb N$.
Elementary estimations then yield
 \begin{align*}
 \mathscr H_{d}(\Omega)=\int_{y\in\Omega}\,d\mathscr H_{d}(y)\leq \sum_{k\in\mathbb N}\int_{\Omega\cap X_k}\,d\mathscr H_{d}(y)=0.
 \end{align*}
Hence $\psi_R$ is geodesically differentiable $\mathscr H_d$-a.e. and in particular $\mu$-a.e. on $X$ since $\mu\ll\mathscr H_d$.

Next we claim that the decreasing sets $\{\psi<\psi_R\}$ in $R$ have $\mu$-negligible intersection $S=\bigcap_{R>0}\{\psi<\psi_R\}$. Indeed if $\mu(S)>0$ then for every $x\in S$ we would have $\psi(x)<\psi_R(x)$ for all $R>0$ and by Lemma \ref{l:Kantorovich} we have $y(x)\in S$ such that $\psi^c(y(x))-\psi(x)=c(x,y(x))$. But then this would imply $\psi^c(y(x))-c(x,y(x))<\psi_R(x)$ for all $R>0$, which for sufficiently large $R>0$ would be impossible as per definition of $\psi_R$. Consequently for large enough $R>0$ it holds $\psi(x)=\psi^c(y)-c(x,y)\geq \psi_R(x)$ for $\mu$-a.e.. On the other hand we have $\psi_R(x)\geq \inf_{y\in X}(\psi^c(y)-c(x,y))=\psi(x)$ for $\mu$-a.a. $x\in X$, consequently $\psi(x)=\psi_R(x)$ for $\mu$-a.a. $x\in X$ for large enough $R>0$. Hence $\{\psi\neq \psi_R\}$ is a $\mu$-negligible set and in particular $\{\psi=\psi_R\}$ has density $1$ at every $x\in X$ with respect to $\mu$. 
Then in view of Definition \ref{d:approx-derivative} the $\mu$-approximate geodesic derivative of $\psi$ exists and it is given by $\widetilde D\psi(x;\gamma)=D\psi_R(x;\gamma)$ for $\mu$-a.a. $x\in X$ and $\gamma\in\Gamma_x(X)$.
 
 The minimal value of the function $z\mapsto \psi(z)+d(z,y)^2$ equals $\psi^c(y)$ and it is attained at $z=x$. By Theorem \ref{th:a.e.-extension} the local geodesic extension property holds $\mathscr H_d$-a.e., then applying Lemma \ref{l:min-max} yields $\widetilde D\psi(x;\gamma)+D_xc(x,y;\gamma)= 0\,(\dagger)$ for $\mu$-a.a. $x\in X$. Since $X$ satisfies the positive angles property $\mathscr H_{d}-$a.e., then by virtue of Lemma \ref{l:twist} there is a unique $y$ locally at $x$ $\mu$-a.e. for which $(\dagger)$ is fulfilled for all $\gamma\in\Gamma_x(X)$. We set $y:= T(x)$, i.e. $\widetilde D\psi(x;\gamma)+D_xc(x,T(x);\gamma)= 0$ for $\mu$-a.e. $x$ and all $\gamma\in\Gamma_{x}(X)$. Similar arguments as in \cite[Theorem 5.30]{Villani} show that $T$ is Borel measurable and unique up to $\mu$-negligible sets.
 \end{proof}


\subsection{The case of Riemannian manifolds}A special class of $\CAT(0)$ spaces are Riemannian manifolds of non-positive sectional curvature. By \cite[Theorem 1A.6]{Brid} a smooth Riemannian manifold $\mathscr M$ is of curvature $\kappa$ in the sense
of Alexandrov  if and only if the sectional curvature of $\mathscr M$ is $\kappa$. We denote by $\langle\cdot,\cdot\rangle_x$ and $\mathscr T_x\mathscr M$ the Riemannian metric and the tangent space at $x\in \mathscr M$ respectively. Given $\gamma,\eta\in\Gamma_x(\mathscr M)$, the Riemannian angle  $\angle_x(\gamma'(0),\eta'(0))$ at $x$ between the geodesic $\gamma$ and $\eta$ is defined as the Euclidean angle between the tangent vectors $\gamma'(0)$ and $\eta'(0)$ in the tangent space $\mathscr T_x\mathscr M$. By \cite[Corollary IA.7]{Brid} the Alexandrov angle $\alpha_x(\gamma,\eta)$ coincides with the Riemannian angle $\angle_x(\gamma'(0),\eta'(0))$. In what follows an important role plays the exponential mapping $\exp_x:\mathscr T_x\mathscr M\to \mathscr M$ which acts by the formula $\exp_x(tv)=\gamma(t)$ where $\gamma'(0)=v\in\mathscr T_x\mathscr M$ and $t\in[0,1]$. By Hadamard's Theorem, e.g. see \cite[Theorem 3.1, \S7]{doCarmo}, $\exp_x:\mathscr T_x\mathscr M\to \mathscr M$ is a global diffeomorphism, so the differential $(d\exp_x)_v$ is well defined and in particular invertible for all $v\in\mathscr T_x\mathscr M\setminus\{0\}$. 
We present a classical result due to Federer, which will prove useful for our next theorem.

%

\begin{lemma}\cite[Theorem 3.1]{Federer}
\label{l:Federer}
Let $\mathscr M$ and $\mathscr N$ be two smooth Riemannian manifolds with $\dim_H\mathscr M=d\geq k=\dim_H\mathscr N$ and $f:\mathscr M\to\mathscr N$ a Lipschitz mapping, then 
\begin{equation}
\label{eq:Federer}
\int_{\mathscr M}g(x)\,Jf(x)\,d\mathscr H_d(x)=\int_{\mathscr N}\int_{f^{-1}(y)}g(x)\,d\mathscr H_{d-k}(x)\,d\mathscr H_k(y)
\end{equation}
for every $\mathscr H_d$--integrable function $g$. Here $Jf(x)$ denotes the square root of the sum of squares of the determinants of the $k\times k$ minors of the  matrix of the differential $(df)_x$.
\end{lemma}

\begin{thm}
	\label{th:manifold}
	Let $(\mathscr M, \langle\cdot,\cdot\rangle_x)$ be a Riemannian manifold of non-positive curvature.  For any $\mu,\nu\in\mathscr P_2(\mathscr M)$ with $\mu$ absolutely continuous, there is a unique optimal transport plan $\pi\in\Pi(\mu,\nu)$. Moreover $\pi=(\Id,T)_{\#}\mu$ for a Borel measurable map $T:\mathscr M\to \mathscr M$ such that $T(x)\in\partial_c\psi(x)$ where $\psi$ is a $c$-convex function and satisfies $Tx=\exp_x(\widetilde\nabla\psi(x))\;\mu$-a.e.. 
\end{thm}

\begin{proof}
It suffices to prove the conditions (i)(ii) in Theorem \ref{th:1}. 
 Let $\gamma_1,\gamma_2\in\Gamma^0_x(\mathscr M)$ such that $\{x,\gamma_1(1),\gamma_2(1)\}$ are non-collinear. Suppose that $\alpha_x(\gamma_1,\gamma_2)=0$, then $\angle_x(\gamma'_1(0),\gamma_2'(0))=0$ in $\mathbb R^d$ where $d=\dim_H\mathscr M$. There exists $a>0$ such that $\gamma_2'(0)=a\,\gamma'_1(0)$. This implies that $\gamma_1(t)=\gamma_2(t/a)$ for $t\in[0,1]$ if $a\geq 1$ and $\gamma_1(a\,t)=\gamma_2(t)$ for $t\in[0,1]$ if $a<1$. In either case we would have $\gamma_i\subseteq\gamma_j$ for $i\neq j$, implying that $\{x,\gamma_1(1),\gamma_2(1)\}$ are collinear. This is impossible. Moreover note that for any $n\in\mathbb N$ we have that $r_x>n$ for all $x\in\mathscr M$, i.e. $\inf\{r_x\,:\,x\in\mathscr M\}>0$ trivially. 
 
 Last we show that disintegration regularity is also satisfied. Let $x\in\mathscr M$ and $\gamma\in\Gamma^0_x(\mathscr M)$.  
 Take $y\in\mathscr M$ and denote $F(y)=\exp_x(A_{\gamma}\,\exp^{-1}_x(y))$  where $\eta\in\Gamma^0_x(\mathscr M)$ satisfies $\eta(1)=y$ and $A_{\gamma}:\mathscr T_x\mathscr M\to \mathscr T_x\mathscr M$ acts by the formula $A_{\gamma}(v)=\gamma'(0)|v|/|\gamma'(0)|$ if $|v|<|\gamma'(0)|$ and $A_{\gamma}(v)=\gamma'(0)$ otherwise. We claim that $F(y)=R_{\gamma}(y)$. Note that $\exp^{-1}_x(y)=\eta'(0)$
 so $A_{\gamma}(\eta'(0))=\gamma'(0)\,|\eta'(0)|/|\gamma'(0)|$ if $|\eta'(0)|<|\gamma'(0)|$ and $A_{\gamma}(\eta'(0))=\gamma'(0)$ otherwise. Consequently we obtain $\exp_x(A_{\gamma}(\eta'(0)))=\gamma(|\eta'(0)|/|\gamma'(0)|)$ if $|\eta'(0)|<|\gamma'(0)|$ and $\exp_x(A_{\gamma}(\eta'(0)))=\gamma(1)$ otherwise. Realizing that $|\gamma'(0)|=\ell(\gamma)$ and $|\eta'(0)|=\ell(\eta)$ we get $d(x,F(y))=d(\gamma(0),\gamma(\ell(\eta)/\ell(\gamma))=\ell(\eta)=d(x,y)$ if $\ell(\eta)<\ell(\gamma)$ and $d(x,F(y))=d(\gamma(0),\gamma(1))$ otherwise. Because $F(y)\in\gamma$ then $F(y)=R_{\gamma}(y)$. Let $\{x_k\}_{k\in\mathbb N}$ be a dense subset of $\mathscr M$ and $r>0$. We can write $\mathscr M=\bigcup_{k\in\mathbb N}B_k$ where $B_k\coloneqq B(x_k,r)$. Let $d_k=\dim_HB_k$. Applying \eqref{eq:absolute} together with Lemma \ref{l:Federer} by identifying $\mathscr N$ with some $\gamma_k\in\mathscr M_{x_k}(B_k)$, we have that $\zeta_{\gamma_k}(y)=JF(y),\;\mathscr H_{d_k}$--a.e. on $B_k$.  Therefore it is enough to show that $JF(y)>0,\;\mathscr H_{d_k}$--a.e. on $B_k$. From the definition of $J$ we then get
 \begin{align*}
 JF(y)=(\sum_{1\leq i,j\leq d}((dF)_{y})_{i,j}^2)^{1/2}=\|(dF)_y\|_{\mathscr F},
 \end{align*}
 where $\|A\|_{\mathscr F}:=(\sum_{i=1}^m\sum_{j=1}^n|a_{ij}|^2)^{1/2}$ denotes the Frobenius norm of an $m\times n$ matrix $A$.
From the chain rule for differentiation we have $$(dF)_y=(d\exp_{x_k})_{A_{\gamma_k}\exp^{-1}_{x_k}(y)}\cdot (dA_{\gamma_k})_{\exp^{-1}_{x_k}(y)}\cdot (d\exp^{-1}_{x_k})_{y}.$$ 
Non-positive curvature implies that $(d\exp_x)_v$ is invertible for all $v\in\mathscr T_x\mathscr M\setminus\{0\}$ and in particular $(d\exp^{-1}_{x_k})_y=(d\exp_{x_k})_{\eta'(0)}^{-1}$. We can write
$$(dF)_{\exp_{x_k}(\eta'(0))}=(d\exp_{x_k})_{A_{\gamma_k}\eta'(0)}\cdot (dA_{\gamma_k})_{\eta'(0)}\cdot (d\exp_{x_k})_{\eta'(0)}^{-1}$$
which in turn yields
$$(dA_{\gamma})_{\eta'(0)}=(d\exp_{x_k})^{-1}_{A_{\gamma_k}\eta'(0)}\cdot(dF)_{\exp_{x_k}(\eta'(0))}\cdot (d\exp_{x_k})_{\eta'(0)}$$ 
consequently 
$$\|(dA_{\gamma_k})_{\eta'(0)}\|_{\mathscr F}\leq\|(d\exp_x)^{-1}_{A_{\gamma_k}\eta'(0)}\|_{\mathscr F}\cdot\|(dF)_{\exp_x(\eta'(0))}\|_{\mathscr F}\cdot \|(d\exp_x)_{\eta'(0)}\|_{\mathscr F}.$$ 
Note that $\|(d\exp_x)_v\|_{\mathscr F}>0$ for all $v\neq 0$. Moreover $\|(dA_{\gamma})_v\|_{\mathscr F}=1/|\gamma'(0)|^{d-1}|v|^{d-1}>0$ for every $v\in\mathscr T_x\mathscr M\setminus\{0\}$. Therefore 
 $$JF(\exp_{x_k}(\eta'(0)))\geq \frac{\|(dA_{\gamma_k})_{\eta'(0)}\|_{\mathscr F}}{\|(d\exp_{x_k})^{-1}_{A_{\gamma_k}\eta'(0)}\|_{\mathscr F}\|(d\exp_{x_k})_{\eta'(0)}\|_{\mathscr F}}>0,\;\eta'(0)\in \mathcal T_{x_k}\mathscr M\setminus\{0\}$$
 equivalently $JF(y)>0,\;y\neq x_k$, i.e. $\zeta_{\gamma_k}(y)>0$ for all $y\in B_k\setminus\{x_k\}$. Let $\psi_R$ be defined as in the proof of Theorem \ref{th:1}, then we have that $\psi_R$ is $\mu$-a.e. geodesically differentiable on $\mathscr M$. Let $x$ be a point in $\mathscr M$ where $\psi_R$ is differentiable and $\textbf u:U\to\mathscr M$ with $U\subseteq \mathbb R^d$ be a parameterization at $x$. For a geodesic $\gamma\in\Gamma_x^0(\mathscr M)$ we then have $\textbf u^{-1}(\gamma(t))=(u_1(t),u_2(t),\cdots,u_d(t))$. Restricting $\psi_R$ to $\gamma$ yields 
 $$\frac{d}{dt}\psi_R(\gamma(t))\Big|_{t=0}=\sum_{i=1}^du_i'(0)\,\frac{\partial \psi_R}{\partial u_i}=\langle \nabla \psi_R(x),\gamma'(0)\rangle_x$$
 which together with  $D\psi_R(x;\gamma)=\frac{d}{dt}\psi_R(\gamma(t))\Big|_{t=0}$ implies
 $D\psi_R(x;\gamma)=\langle \nabla \psi_R(x),\gamma'(0)\rangle_x.$
 By similar arguments we obtain that
 $D_xc(x,Tx;\gamma)=\langle \nabla c(x,T(x)),\gamma'(0)\rangle_x.$
 On the other hand elementary calculations show that $\nabla c(x,Tx)=-\exp^{-1}_xTx$, so
 \begin{equation*}
 	\langle\nabla\psi_R(x),\gamma'(0)\rangle_x-\langle \exp^{-1}_xTx,\gamma'(0)\rangle_x=0,\quad\forall \gamma\in\Gamma^0_x(\mathscr M),\;\mu-a.e..
 	\end{equation*} 
 Taking $\gamma\in\Gamma_x^0(\mathscr M)$ with $\gamma'(0)=\nabla\psi(x)-\exp^{-1}_xTx$ and $\widetilde\nabla\psi=\nabla\psi_R$ yields the claim. 	
\end{proof}

\begin{remark}
\label{r:manifold}
Theorem \ref{th:manifold} is a corollary of the well-known result of \cite[Theorem 1]{Feldman} applied to manifolds with non-positive curvature. 
\end{remark}

\subsection{The case of polyhedral complexes} Another class of interesting spaces are the Euclidean polyhedral complexes that are  $\CAT(0)$ spaces. Such a space is built from a finite collection of Euclidean polytopes glued in a way satisfying certain conditions, e.g. see \cite[Definition 7.37]{Brid}. Not all Euclidean polyhedral complexes are $\CAT(0)$ spaces. A theorem of Gromov, e.g. see \cite[Theorem 5.18, \S2]{Brid}, provides necessary and sufficient condition for when a Euclidean polyhedral complex is a $\CAT(0)$ space.
As a byproduct of Theorem \ref{th:1} we obtain the following result for such polyhedral complexes.

\begin{thm}
\label{th:poly-complex}
Let $X$ be a $\CAT(0)$ polyhedral complex and $\mu,\nu\in\mathscr P_2(X)$ with $\mu$ absolutely continuous, then Theorem \ref{th:1} holds true.
\end{thm} 
\begin{proof}
Let $d=\max\{\dim_HP\,:\,P\in\mathcal C\}$ where $\mathcal C$ is a collection of Euclidean polytopes $P$.
It is enough to show that the three conditions hold $\mathscr H_d$-a.e., consequently only for interiors of polytopes $P$ such that $\dim_HP=d$. The positive angles property holds trivially in the interior of a polytope, since one can think of it as a (usual) convex set in its Euclidean ambient space. To prove the last property take any $x\in P$ an interior point and consider $\gamma\in\Gamma_x^0(P)$ for which $\ell(\gamma)=\max\{d(x,p)\,:\,p\;\text{is a vertex of}\,P\}$. This maximal value exists since there are at most a finite number of vertices for each polytope $P$.  Moreover $d(x,y)\leq \ell(\gamma)$ for all $y\in P$ implies in particular that $R_{\gamma}(y)$ is an interior point of $\gamma$ for all $y\in P$, except possibly for some of the vertices of $P$. It is not difficult to show that $$R_{\gamma}(y)=\frac{p_{\gamma}-x}{\|p_{\gamma}-x\|_{\mathbb E}}\,\|y-x\|_{\mathbb E}, \quad\text{for all}\;y\in P\setminus\{\text{vertices of}\;P\}.$$
Here $\mathbb E$ denotes the Euclidean ambient space of $P$ and $p_{\gamma}$ is the vertex of $P$ that realizes $\max\{d(x,p)\,:\,p\;\text{is a vertex of}\,P\}$ and such that $\gamma$ connects $x$ with $p_{\gamma}$. Again an application of Lemma \ref{l:Federer} now with $F=R_{\gamma}$ and \eqref{eq:absolute} yields $JR_{\gamma}=\zeta_{\gamma},\,\mathscr H_d$-a.e. on $P$. It suffices to show that $JR_{\gamma}>0,\,\mathscr H_d$-a.e. on $P$. Notice that $JR_{\gamma}(y)=\|(dR_{\gamma})_y\|_{\mathscr F}$. Elementary calculations show that $$\|(dR_{\gamma})_y\|_{\mathscr F}=\Big(\frac{1}{\|p_{\gamma}-x\|_{\mathbb E}\,\|y-x\|_{\mathbb E}}\Big)^{d-1}>0,\;\text{for all}\;y\in P\setminus\{x\}.$$ 
Thus $\zeta_{\gamma}(y)>0$ for all $y\in P\setminus\{x\}$, and since $x\in P$ is an arbitrary interior point, then $X$ satisfies the disintegration property on $P$ for $\mathscr H_d$-a.a. $x\in P$. Since $P$ is arbitrary then $X$ satisfies the disintegration property $\mathscr H_d$-a.e.. This completes the proof. 
\end{proof}

\section{A polar factorization theorem}
\label{s:polar}

\begin{lemma}
\label{l:invertible}
Let $(X,d), \mu,\nu,\psi, T$ be as in Theorem \ref{th:1} and  conditions $(i),(ii)$ be satisfied. If additionally $\nu\ll\mathscr H_d$, then there exists $T^{*}:X\to X$ such that $\pi^*=(\Id,T^{*}_{\#})\nu$ is the unique optimal transport plan from $\nu$ to $\mu$ and $T^*Tx=x$ for $\mu$-a.a. $x\in X$. In particular $T$ and $T^*$ are inverse maps of each other.
\end{lemma}
\begin{proof}
The existence and uniqueness of $T^*$ follows directly from Theorem \ref{th:1} since all conditions hold by assumption. We need only show that $T^*Tx=x$ for all $x\in X$. By Theorem \ref{th:1} again it follows that $T^*(y)\in\partial_c\psi^*(y),\;\nu$-a.e. for some $c$-convex function $\psi^*$. Note by Corollary \ref{l:Kantorovich} we have that $\partial_c\psi^*(y)$ consists of a single element and that $(\psi^*)^c(T^*y)-\psi^*(y)=c(y,T^*y)$ for $\nu$-a.a. $y\in X$. 
In view of \cite[Remark 5.15]{Villani} we have that $\psi^*=-\psi^c$ and $(\psi^*)^c=-\psi$, consequently $\psi^c(y)-\psi(T^*y)=c(y,T^*y)$ for $\nu$-a.a. $y\in X$. This implies in particular that $\partial_c\psi(T^*y)=\{y\}$ for $\nu$-a.a. $y\in X$. But $\Gamma=\{(x, Tx)\,:\,x\in X\}\subseteq \partial_c\psi$ and $\pi(\partial_c\psi\setminus\Gamma)=0$ imply that $(T^*y,y)\in\Gamma$ for $\nu$-a.a. $y\in X$, that is $y=Tx$ and $T^*y=x$ for $\mu$-a.a. $x\in X$. Therefore $T^*Tx=x$ for $\mu$-a.a. $x\in X$. Similarly $TT^*y=y$ for $\nu$-a.a. $y\in X$. We conclude that $T$ and $T^*$ are inverse maps of each other. This completes the proof.
\end{proof}

\begin{thm}
\label{th:polarization}
Let $(X,d)$ be a $\CAT(0)$ space as in Theorem \ref{th:1}. Consider a Borel map $s:X\to X$ and a probability measure $\mu\ll\mathscr H_d$. If $s_{\#}\mu\ll\mathscr H_d$ then $s=T\circ u$ where $T$ is the unique optimal transport map between $\mu$ and $s_{\#}\mu$ and $u$ satisfies $u_{\#}\mu=\mu$.  The factoring maps $T,u$ are unique up to $\mu$-negligible sets in $X$.
\end{thm}

\begin{proof}
First note that $T_{\#}\mu=s_{\#}\mu$ and by Lemma \ref{l:invertible} there exists a map $T^*:X\to X$ with $T^*_{\#}(s_{\#}\mu)=\mu$ and $T^*Tx=x$ for $\mu$-a.a. $x\in X$ that is the unique optimal transport mapping from $s_{\#}\mu$ to $\mu$. Define $u=T^*\circ s,\,\mu$-a.a. $x\in X$. Note that for any $\mu$-integrable function $f$ on $X$ we have
\begin{align*}
\int_Xf(x)\,d\mu(x)=\int_Xf(x)\,d(T^*_{\#}s_{\#}\mu)(x)=\int_Xf(x)\,d(u_{\#}\mu)(x).
%
\end{align*}
Consequently $u_{\#}\mu=\mu$. To show uniqueness, let $T',u'$ be such that $s=T'\circ u'$. Evidently $T=T',\,\mu$-a.e. on $X$ by Theorem \ref{th:1}. Therefore $s(x)=T(u'(x))$ for $\mu$-a.a. $x\in X$. But then $u(x)=T^*(s(x))=T^*(T(u'(x)))=u'(x)$ for $\mu$-a.a. $x\in X$. 
\end{proof}

\bibliographystyle{plain}
\bibliography{references}

%
%
%

\end{document}